\documentclass[12pt]{amsart}

\usepackage[english]{babel}
\usepackage{tikz-cd}
\usepackage{amsrefs}
\usepackage{graphicx}
\graphicspath{ {images/} }
\usepackage[T1]{fontenc}
\usepackage{txfonts}
\usepackage{times}
\usepackage{amssymb,verbatim,amscd,amsfonts,amsmath,amsthm,enumerate}
\usepackage[all]{xy}
\usepackage{subfig}
\usepackage{tikz}
\usetikzlibrary{shapes,arrows,shadows}
\usetikzlibrary{decorations.markings}
\usepackage{color}
\usepackage[normalem]{ulem}
\usepackage{hyperref}
\usepackage{wrapfig}
\usetikzlibrary{arrows}
\usepackage{pb-diagram}

\usetikzlibrary{shapes.geometric}

\newtheorem{theorem}{Theorem}[section]

\newtheorem{corollary}[theorem]{Corollary}
\newtheorem{lemma}[theorem]{Lemma}

\newtheorem{remark}[theorem]{Remark}
\newtheorem{construction}{Construction}[section]

\newtheorem{conjecture}[theorem]{Conjecture}
\newtheorem{definition}[theorem]{Definition}
\newtheorem{question}[theorem]{Question}

\newtheorem*{thma}{Theorem A}{\bf}{\it}
\newtheorem*{thmb}{Theorem B}{\bf}{\it}
{\bf}{\it}

\newcommand{\N}{\mathbb{N}}

\newcommand{\Z}{\mathbb{Z}}
\newcommand{\C}{\mathbb{C}}

\newcommand{\sing}{\text{Sing}}
\newcommand{\kinf}{\mathbf{k}_{\infty}}
\newcommand{\inv}{^{-1}}

\begin{document}
	
	%%%%%%%%%%%%%%%%%%%%%%%%%%%%%%%%%%%%%%%
	%%%%%%%%%%%%%%%%%%%%%%%%%%
	\renewcommand{\tablename}{Diagram}
	
	%%%%%%%%%%%%%%%%%%%%%%%%%%%%%%%%%%%
	%%%%%%%%%%%%%%%%%%%%%%%%%%%%%%%%%%%
	%%%%%%Authors
	%%%%%%%%%%%%%%%%%%%%%%%%%%%%%%%%%%
	%%%%%%%%%%%%%%%%%%%%%%%%%%%%%%%%%%
	\title{Stratification and Realization of Singularity data on the Loch Ness monster}
	\author{Nate Fisher and Camilo Ram\'irez Maluendas}
	\address{Department of Mathematics and Statistics, Swarthmore College, Swarthmore, PA, USA}
	\email{nfisher2@swarthmore.edu}
	\address{Universidad Nacional de Colombia, Sede Manizales, Facultad de Ciencias Exactas y Naturales, Departamento de Matemáticas y Estadística, Cr. 27 \# 64-60, Manizales, CP 170004, Caldas, Colombia}
	\email{camramirezma@unal.edu.co}

	%%%%%%%%%%%%%%%%%%%%%%%%%%%%%%%%%%%%%%%%%%%
	%%%%%%%%%%%%%%%%%%%%%%%%%%%%%%%%%%%%%%%%%%%
	%%%%%%%%%%%Abstract
	%%%%%%%%%%%%%%%%%%%%%%%%%%%%%%%%%%%%%%%%%%%
	%%%%%%%%%%%%%%%%%%%%%%%%%%%%%%%%%%%%%%%%%%%
	
	\begin{abstract}
	The classical theory of compact translation surfaces is organized through the stratification of moduli spaces according to the orders of the zeros of the associated Abelian differentials. In the setting of tame translation surfaces of infinite topological type, no analogous notion of strata has been systematically developed. 
    In this article, we introduce a stratification theory for tame translation surfaces based on their singularity data. To each non-compact tame translation surface, we associate a sequence recording the cardinalities of finite-angle cone singularities of each possible cone angle together with the cardinality of the set of infinite-angle singularities. We call this sequence the \emph{singularity data} of the surface.

We study the realization problem for singularity data on the Loch Ness monster. Our first main result shows that there is no tame translation structure on the Loch Ness monster having only a finite number of finite cone angle singularities. We then prove that this is the only obstruction: every singularity data not excluded by the previous result is realized by a tame translation structure on the Loch Ness monster. As a consequence, we obtain a complete characterization of the strata of tame translation structures on this surface in terms of their singularity data.

These results provide the first realization theorem for strata of non-compact tame translation surfaces and reveal a strong interaction between singularity data and the ends space.
	\end{abstract}

    	\maketitle
	%\textbf{Key words.} 

	%\tableofcontents
	
	%%%%%%%%%%%%%%%%%%%%%%%%%%%%%%%%%%%%%%%
	%%%%%%%%%%%%%%%%%%%%%%%%%%%%%%%%%%%%%%%
	%%%%%%%%%%%Introduction
	%%%%%%%%%%%%%%%%%%%%%%%%%%%%%%%%%%%%%%%
	%%%%%%%%%%%%%%%%%%%%%%%%%%%%%%%%%%%%%%

\section{Introduction}
Translation surfaces have become an important object of study in several branches of mathematics, including Teichmüller theory, dynamical systems, geometric topology, algebraic geometry, billiards in rational polygons and others. See, for example, \cites{Gut-judge, maluendas2025ends, apisa2021marked}. Several authors have focused their interest on the study of the stratification of the moduli spaces of compact translation surfaces \cites{Veech1990, EskinOkounkov2001, Kont-Zor, ChenMoller2020, Lanneau2008, McMullen2007}. Given a compact translation surface, its singularities correspond to the zeros of an Abelian differential and are necessarily finite in number. If the cone angles are
\[
2\pi(k_{1}+1),\ldots,2\pi({k}_{l}+1),
\]
then the corresponding translation surface has genus $g$ and belongs to the stratum
\[
\mathcal H(k_{1},\ldots,k_{l}),
\]
where the integers $k_{1},\ldots,k_{l}$ satisfy the Gauss--Bonnet relation
\[
k_{1}+\cdots+k_{l}=2g-2.
\]
However, in the theory of translation surfaces of infinite topological type, there is as of yet no description of a stratification similar to that of the compact case.

During the last decade, several mathematicians have published interesting contributions on non-compact translation surfaces \cites{Maluendas2025, HidMo2023, randecker2018wild, hooper2019}. A natural framework for studying non-compact translation surfaces is provided by the class of \emph{tame translation surfaces}, introduced by F. Valdez \cite{Val2012}. Roughly speaking, a translation surface is tame if every point of its metric completion is either a regular point, a finite-angle cone singularity, or an infinite-angle cone singularity. 

The first objective of this article is to introduce a natural stratification theory for tame translation surfaces. Let $(S,\mathcal A)$ be a tame translation surface, as defined in Section~\ref{subsection_tame_translation_surface} below. For each $n\in\mathbb N$, let ${\rm Sing}_{n}(S)$ denote the set of cone singularities of angle $2\pi(n+1)$ and let ${\rm Sing}_0(S)$ denote the set of infinite-angle singularities. Writing
\[
k_{n}:=\operatorname{card}({\rm Sing}_{n}(S)),
\]
we obtain a sequence
\[
\textbf{k}_{\infty}:=(k_n)_{n\in\mathbb N_0},
\]
which records the number of singularities of each possible cone angle. We call this sequence the \emph{singularity data} of $(S,\mathcal A)$.

This invariant may be viewed as a natural extension of the classical partition $(k_1,\ldots,k_r)$ defining the strata of compact translation surfaces. Indeed, in the compact case all but finitely many terms of the sequence are 0, and the singularity data completely determines the genus and the associated stratum of the surface. Motivated by this observation, we define the strata of tame translation surfaces as the collections of tame translation structures having the same singularity data.
Once this definition is established, a fundamental realization problem immediately arises.

\begin{quote}
\textit{Which singularity data can be realized by tame translation structures on a given non-compact surface?}
\end{quote}
Recall that orientable non-compact surfaces are classified, up
to homeomorphisms, by their genus $g(S)\in\mathbb{N}_{0} \cup \{\infty\}$, and a pair of nested topological spaces
${\rm Ends}_{\infty}(S)  \subseteq {\rm Ends}(S)$ homeomorphic to a pair of nested closed subsets of the Cantor space. See Section~\ref{subsection:surfaces} below. In the present article we answer the previous question the simplest surface of infinite genus, which is well-known as the \emph{Loch Ness monster}. Up to homeomorphism, it is the unique orientable surface of infinite genus with exactly one end. See Figure \ref{Fig:LNM}.
    \begin{figure}[!ht]
	\begin{center}	
		\begin{tikzpicture}[baseline=(current bounding box.north)]
		\begin{scope}[scale=0.7]
		\clip (-6,-1.5) rectangle (8,2);
		\draw [line width=1pt] (-2.5,0) to[out=90,in=180] (-1.5,1.5);
		\draw [line width=1pt] (-1.5,1.5) to[out=0,in=90] (-0.5,0);
		\draw [line width=1pt] (-2,0) to[out=90,in=180] (-1.5,1);
		\draw [line width=1pt] (-1.5,1) to[out=0,in=90] (-1,0);
		%%%%%%%%%%%%%%%%%%%%%%%%%%%%%%%%%%%%%%%%%%
		%%%%%%%%%%%%%%%%%%%%%%%%%%%%%%%%%%%%%%%%%%%%
		\draw [dashed, line width=0.6pt]  (-2.25,0) ellipse (2.5mm and 1mm);
		\draw [line width=1pt] (-2.5,0) arc
		[
		start angle=180,
		end angle=360,
		x radius=2.5mm,
		y radius =1mm
		] ;
		\draw [dashed, line width=0.6pt]  (-0.75,0) ellipse (2.5mm and 1mm);
		\draw [line width=1pt] (-1,0) arc
		[
		start angle=180,
		end angle=360,
		x radius=2.5mm,
		y radius =1mm
		] ;
		\draw [dashed, line width=0.6pt]  (-1.5,1.25) ellipse (1mm and 2.5mm);
		\draw [line width=1pt] (-1.5,1.5) arc
		[
		start angle=90,
		end angle=270,
		x radius=1mm,
		y radius =2.5mm
		] ;
		%%%%%%%%%%%%%%%
		\draw [line width=1pt] (0,0) to[out=90,in=180] (1,1.5);
		\draw [line width=1pt] (1,1.5) to[out=0,in=90] (2,0);
		\draw [line width=1pt] (0.5,0) to[out=90,in=180] (1,1);
		\draw [line width=1pt] (1,1) to[out=0,in=90] (1.5,0);
		%%%%%%%%%%%%%%%%%%%%%%
		%%%%%%%%%%%
		\draw [dashed, line width=0.6pt]  (0.25,0) ellipse (2.5mm and 1mm);
		\draw [line width=1pt] (0,0) arc
		[
		start angle=180,
		end angle=360,
		x radius=2.5mm,
		y radius =1mm
		] ;
		\draw [dashed, line width=0.6pt]  (1.75,0) ellipse (2.5mm and 1mm);
		\draw [line width=1pt] (1.5,0) arc
		[
		start angle=180,
		end angle=360,
		x radius=2.5mm,
		y radius =1mm
		] ;
		\draw [dashed, line width=0.6pt]  (1,1.25) ellipse (1mm and 2.5mm);
		\draw [line width=1pt] (1,1.5) arc
		[
		start angle=90,
		end angle=270,
		x radius=1mm,
		y radius =2.5mm
		] ;
		%%%%%%%%%%%%%%%%%%%%%%%%%%%%
		%%%%%%%%%%%%
		\draw [line width=1pt] (2.5,0) to[out=90,in=180] (3.5,1.5);
		\draw [line width=1pt] (3.5,1.5) to[out=0,in=90] (4.5,0);
		\draw [line width=1pt] (3,0) to[out=90,in=180] (3.5,1);
		\draw [line width=1pt] (3.5,1) to[out=0,in=90] (4,0);
		%%%%%%%%%%%%%%%%%%%%%%%%
		%%%%%%%%%
		\draw [dashed, line width=0.6pt]  (2.75,0) ellipse (2.5mm and 1mm);
		\draw [line width=1pt] (2.5,0) arc
		[
		start angle=180,
		end angle=360,
		x radius=2.5mm,
		y radius =1mm
		] ;
		\draw [dashed, line width=0.6pt]  (4.25,0) ellipse (2.5mm and 1mm);
		\draw [line width=1pt] (4,0) arc
		[
		start angle=180,
		end angle=360,
		x radius=2.5mm,
		y radius =1mm
		] ;
		\draw [dashed, line width=0.6pt]  (3.5,1.25) ellipse (1mm and 2.5mm);
		\draw [line width=1pt] (3.5,1.5) arc
		[
		start angle=90,
		end angle=270,
		x radius=1mm,
		y radius =2.5mm
		] ;
		%%%%%%%%%%%%%%%%%%%%%%%%%%%
	%%%%%%%%%%%%%%%%%%%%%%
		%%%%%%%%%%%%%%%%%%%%%%%%%%%%%%%%%%%
		\draw [line width=1pt](-3,-0.8) -- (5,-0.8);
		\draw [line width=1pt](-3,0.6) -- (-2.45,0.6);
		\draw [line width=1pt](-1.95,0.6) -- (-1.05,0.6);
		\draw [line width=1pt](-0.55,0.6) -- (0.05,0.6);
		\draw [line width=1pt](0.55,0.6) -- (1.45,0.6);
		\draw [line width=1pt](1.91,0.6) -- (2.55,0.6);
		\draw [line width=1pt](3.05,0.6) -- (3.95,0.6);
		\draw [line width=1pt](4.45,0.6) -- (5,0.6);
		\draw [line width=1pt](-3,-0.825) -- (-3,0.625);
		\draw [line width=1pt](5,-0.825) -- (5,0.625);
		%%%%%%%%%%%%%%%%
		\node at (5.5,0) {$\ldots$};
		\node at (-3.5,0) {$\ldots$};
		%\node at (1,1.5) {$\vdots$};
	\node at (1.1,-1.1) {$\vdots$};
		\end{scope}
		\end{tikzpicture} 
				\caption{\em Loch Ness monster.}
	\label{Fig:LNM}
	\end{center}	
\end{figure}
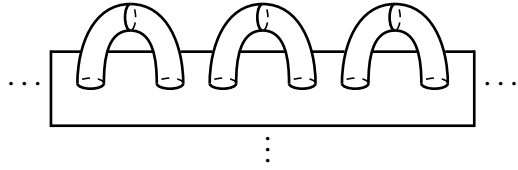 

Our first main result shows that, unlike in the compact case, finite collections of finite-angle singularities cannot occur on tame translation structures carried by the Loch Ness monster.

\begin{thma}\label{thma}
There does not exist a tame translation structure on the Loch Ness monster having singularity data $\kinf$ where $k_0 = 0$ and $\sum_{n=1}^\infty k_n < \infty$. That is, a tame translation structure on the Loch Ness monster must have at least one infinite-angle cone singularities or infinitely many finite-angle cone singularities.
\end{thma}

Theorem \textbf{A} reveals a new global obstruction relating the topology of the underlying surface and the singularity structure of a tame translation surface. In particular, it shows that the classical philosophy of prescribing arbitrary finite singularity data fails dramatically in the non-compact surfaces setting.

%If $S$ is a tame translation surface, for each $n\in\mathbb{N}$, let ${\rm Sing}_{n}(S)$ be the set of all singularities of cone angle $2\pi (n+1)$ of $S$. Let ${\rm Sing}_{0}(S)$ denote the set of all singularities of infinity cone angle of $S$. The set of all finite and infinite cone-angle singularities of $\widehat{S}$ is written in the form
%\[
%{\rm Sing}(\widehat{S})=\bigsqcup\limits_{n=0}^{\infty} {\rm Sing}_{n}(S).
%\]
%For each $n\in\mathbb{N}_{0}$, set 
%\[
%{\rm card}({\rm Sing}_{n}(S)):=k_{n},
%\]
%where $0\leq k_{n}\leq \aleph_{0}$. The sequence 
%\[
%\mathbf{k}_{\infty}:=(k_{n})_{n\in\mathbb{N}_{0}}
%\]
%is called the \emph{singularity data} of $S$.

Our second result proves that this obstruction is essentially the only one. 

\begin{thmb}\label{thmb}
Let
$\textbf{k}_{\infty}=(k_n)_{n\in\mathbb N_0}$ 
represent singularity data not excluded by Theorem \textbf{A}. Then there exists a tame translation structure on the Loch Ness monster whose singularity data is exactly 
$\textbf{k}_{\infty}=(k_n)_{n\in\mathbb{N}_{0}}$.
\end{thmb}

Consequently, every admissible singularity data can be realized by a tame translation structure on the Loch Ness monster. Together, Theorems \textbf{A} and \textbf{B} provide a complete characterization of the strata of tame translation structures on this surface.

From a broader perspective, our results suggest that singularity data play for tame translation surfaces a role analogous to that played by partitions of $2g-2$ in the classical theory of Abelian differentials. However, the realization theory is governed not only by local singularity data but also by the ends spaces of the underlying surface. This interaction between singularities and ends appears to be one of the fundamental features distinguishing non-compact translation surfaces from their compact counterparts.

The proofs our main theorems combine techniques from the topology of non-compact surfaces, the theory of tame translation surfaces, and explicit geometric constructions based on infinite gluings of marked translation surfaces. A key ingredient is the construction of families of tame translation structures that allow complete control over the number and type of singularities while preserving the topological type of the Loch Ness monster.

\vspace{.5cm}
\noindent\textbf{Organization of the paper.}
In Section~\ref{sec:preliminaries}, we give background information on topological surfaces, the space of ends, and tame translation structures. In Section~\ref{sec:proof_thma} we prove Theorem \textbf{A}, showing there is no tame translation structure on the Loch Ness monster containing only a finite number of cone point singularities, all of finite cone angle. In Section~\ref{sec:proof_thmb} we develop auxiliary constructions which will help us build tame translation structures on the Loch Ness monster with specific singularity data. We then prove Theorem \textbf{B}. Finally, we end in Section~\ref{sec:future} by sharing open questions, conjectures, and directions for future exploration.

\vspace{.5cm}
\noindent\textbf{Acknowledgments.} The first author gratefully acknowledges support from Fulbright Scholar Program sponsored by the U.S. Department of State and Fulbright Colombia. He also would like to thank the Universidad Nacional de Colombia, Sede Manizales and, in particular, the Departamento de Matem\'aticas y Estad\'istica for being kind and generous hosts for the duration of his Fulbright grant.

The second author expresses his gratitude to Universidad Nacional de Colombia, Sede Manizales. He has dedicated this work to his beautiful family: Marbella and Emilio, in appreciation of their love and support. 

% (Save this for later once there are referees) The authors are especially grateful to the anonymous referees for their invaluable contributions and useful comments. 

\section{Preliminaries}\label{sec:preliminaries}
\subsection{Topological surfaces}\label{subsection:surfaces}
Throughout this paper, a \emph{surface} is a connected $2$-dimensional topological manifold, denoted by $S$. Unless explicitly stated otherwise, all surfaces are assumed to be orientable.

\subsubsection{Space of ends of topological surfaces.} 
Intuitively, an \emph{end} of a surface represents a direction in which the surface escapes to infinity. We briefly recall the definition. 

\begin{definition}[\cite{Fre}] \label{section_ends_space} 
\begin{upshape}
Let $S$ be a connected surface, and let $(U_n)_{n\in\mathbb{N}}$ be an infinite nested sequence 
\[
U_{1}\supset U_{2}\supset\ldots
\]
of non-empty connected open subsets of $S$, satisfying the following conditions:
\begin{enumerate}
\item[\textbf{(1)}]  for every $n\in\mathbb{N}$, the boundary $\partial U_n$ of $U_{n}$ is compact; 

\item[\textbf{(2)}]  $\bigcap\limits_{n\in\mathbb{N}}\overline{U_{n}}=\emptyset
$;

\item[\textbf{(3)}] for every compact  subset $K\subset S$, there exists $m\in\mathbb{N}$ such that $K\cap U_m =\emptyset$.
\end{enumerate}
Two such sequences $(U_n)_{n\in\mathbb{N}}$ and $(U'_{n})_{n\in\mathbb{N}}$ are said to be \emph{equivalent} if for every $n\in\mathbb{N}$, there exist $j,k\in\mathbb{N}$ such that 
\[
U_{n}\supset U'_{j},\qquad \text{and} \qquad U'_{n}\supset U_{k}.
\]
An \emph{end} of $S$ is an equivalence class of such sequences. The set of all ends of $S$ is denoted by ${\rm Ends}(S)$.

The topology on ${\rm Ends}(S)$ is defined as follows. For any non-empty open subset $U\subset S$ with compact boundary $\partial U$, let
\begin{equation}\label{eq:end_open}
U^{*}:=\left\{[U_{n}]_{n\in\mathbb{N}}\in{\rm Ends}(S)\hspace{1mm}|\hspace{1mm}U_{j}\subset U\hspace{1mm}\text{for some }j\in\mathbb{N}\right\}.
\end{equation}
The collection of all such sets $U^{*}$ forms a basis for the topology of ${\rm Ends}(S)$ (see \cite{Fre}*{1. Kapitel}).
\end{upshape}
\end{definition}

\begin{theorem}[\cite{Ray}]
The space ${\rm Ends}(S)$ is Hausdorff, totally disconnected,
and compact.
\end{theorem}

By a \emph{subsurface} of $S$ we mean an embedded surface which is a closed subset of $S$ and whose boundary consists of a finite number of nonintersecting simple closed curves. A subsurface may be either compact or non-compact.
The {\it reduced genus} of a compact subsurface $\tilde{S}\subset S$,  with $q(\tilde{S})$ boundary curves and Euler characteristic $\chi(\tilde{S})$, is the number 
\[
g(\tilde{S})=1-\frac{1}{2}\left(\chi(\tilde{S})+q(\tilde{S})\right).
\] 
The {\it genus} of $S$ is the supremum of the genera of its compact subsurfaces. It may be a non-negative integer or $\infty$. 

A surface $S$ is said to be {\it planar} if it has genus zero, equivalently, it is homeomorphic to an open subset of the complex plane.

\begin{remark}
\begin{upshape}
	By Definition \ref{section_ends_space}, every end can be represented by a sequence $(U_{n})_{n\in\mathbb{N}}$ whose closures $\overline{U}_{n}$ are subsurfaces of $S$. 
    
    An end $[U_n]_{n\in\mathbb{N}}\in {\rm Ends}(S)$ is called \emph{planar}, if there is $l\in\mathbb{N}$ such that the subsurface $\overline{U}_l\subset S$ is planar.
    \end{upshape}
\end{remark}

Let ${\rm Ends}_{\infty}(S)$ denote the subset of ${\rm Ends}(S)$ consisting of all non-planar ends (equivalently, ends carrying infinite genus). It follows directly from the definition that ${\rm Ends}_{\infty}(S)$ is a closed subset of ${\rm Ends}(S)$ (see \cite{Ian}*{p. 261}). Consequently, the triple 
$\left(g,{\rm Ends}_{\infty}(S),{\rm Ends}(S)\right)$,
where $g$ is the genus of $S$, is a topological invariant.

\begin{theorem}[Classification of non-compact surfaces \cites{Ker, Ian}]
Two surfaces $S_1$ and $S_2$  having the same genus, are topologically equivalent if and only if there exists a homeomorphism $f: {\rm Ends}(S_1)\to {\rm Ends}(S_2)$ such that $f( {\rm Ends}_{\infty}(S_1))= {\rm Ends}_{\infty}(S_2)$.
\end{theorem}

\begin{definition}[\cites{PSul,ArRa}]
\begin{upshape}
    The \emph{Loch Ness monster} is the unique, up to homeomorphism, orientable surface of infinite genus with exactly one end; see Figure \ref{Fig:LNM}.
   \end{upshape}
 \end{definition}  

\begin{lemma}[\cite{SPE}*{\S 5.1., p. 320}]\label{lemma:spec} 
	The surface $S$ has exactly one end if and only if for every compact subset $K \subset S$ there is a compact subset $K^{'}\subset S$ such that $K\subset K^{'}$ and $S \setminus  K^{'}$ is connected.
\end{lemma} 

\subsubsection{Pair of pants decomposition.} A simple closed curve on a surface $S$ is called \emph{essential} if it is not homotopic to a point, a puncture, or a boundary component of $S$.

A \emph{pair of pants} (or \emph{$Y$-piece}) is a compact surface homeomorphic to a sphere with three open disks removed. Equivalently, it is a compact surface whose boundary consists of three pairwise disjoint simple closed curves. See Figure \ref{fig:pair_of_pants}.
\begin{figure}[!ht]
    \centering
\begin{tikzpicture}[baseline=(current bounding box.north)]
		\begin{scope}[scale=0.9]
		\clip (-4.4,-0.9) rectangle (0.4,2.6);
		\draw [color=red, line width=1pt] (-4,0) ellipse (2mm and 5mm);
		\draw [color=red, line width=1pt] (-2,2.3) ellipse (5mm and 2mm);
		\draw [color=red, line width=1pt]  (0,0) ellipse (2mm and 5mm);
		\draw [line width=1pt](-4,-0.5) to[out=0,in=-180] (0,-0.5);
		\draw [line width=1pt] (-4,0.5) to[out=0,in=-90] (-2.5,2.3);
		\draw [line width=1pt] (-1.5,2.3) to[out=-90,in=180] (0,0.5);
		\end{scope}
		\end{tikzpicture}
    \caption{\em A  pair of pants or $Y$-piece.}
    \label{fig:pair_of_pants}
\end{figure}
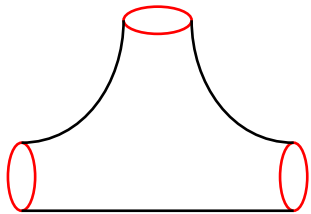

A closed subset of a surface is called a \emph{cylinder} if it is homeomorphic to $S^{1}\times [0,\infty)$, where $S^{1}$ denotes the circle.

\begin{definition}
\begin{upshape}
Let $S$ be a surface. A \emph{pants decomposition} of $S$ is a collection of pairwise disjoint simple closed curves such that cutting $S$ along these curves yields a disjoint union of pairs of pants.

Equivalently, a pants decomposition is a maximal collection of pairwise disjoint, non-isotopic, essential simple closed curves on $S$.
\end{upshape}
\end{definition}

\begin{theorem}[\cite{Alvarez2004}*{Theorem 1}]
    Every orientable surface, except the sphere, the plane, and the torus, admits a decomposition into pairs of pants and cylinders with pairwise disjoint interiors.
\end{theorem}

\begin{remark}
    In the case of the Loch Ness monster surface, no cylinders occur in such a decomposition; the surface is a union of pairs of pants with pairwise disjoint interiors. Indeed, the Loch Ness monster has a single end which is accumulated by genus, while a cylinder in the decomposition would correspond to a planar end of the surface.
\end{remark}

\subsection{Tame translation surfaces}\label{subsection_tame_translation_surface}

Let $S$ be a surface. Unless explicitly stated otherwise, all surfaces are assumed to be without boundary. A \emph{translation structure} on $S$ is a maximal atlas
\[
\mathcal{A}=\{(U_{i},\varphi_{i})_{i\in I}\},
\]
where each chart $\varphi_{i}:U_{i}\to\mathbb{C}$ is a homeomorphism onto an open subset of $\mathbb{C}$, satisfying the following conditions:
    \begin{itemize}
        \item[\textbf{(1)}] The domains of the charts cover $S\setminus {\rm Sing}(S)$, where ${\rm Sing}(S)\subset S$ is a discrete set of points. Every point of ${\rm Sing}(S)$ is non-removable, in the sense that the atlas cannot be extended across it.
        
        \item[\textbf{(2)}] Whenever $(U,\varphi)$ and $(V,\psi)$ are charts with $U\cap V\neq \emptyset$, the transition map
        \[
        \varphi\circ\psi^{-1}: \psi(U\cap V)\to \varphi(U\cap V)
        \]
        is locally a restriction of a translation map, that is
        \[
        \varphi \circ \psi^{-1}(z)=z+c,
        \]
        for a constant $c\in\mathbb{C}$;

        \item[\textbf{(3)}] The atlas $\mathcal{A}$ is maximal with respect to the preceding properties.
    \end{itemize}
The pair $(S,\mathcal A)$ is called a \emph{translation surface}. Whenever the atlas is clear from the context, we simply write $S$ for $(S,\mathcal A)$. We refer the reader to \cites{Ath-Ma,Del-Hub-Val,wright-survey} for background material on translation surfaces.

Every translation surface inherits a natural flat metric by pulling back the complex plane metric via the coordinate charts. Let $\widehat S$ denote the metric completion of $S$ with respect to this flat metric. By the Uniformization Theorem \cite{Abi}*{p. 580}, the only complete translation surfaces (that is, those satisfying $S=\widehat{S}$) are the complex plane, the flat torus, and the flat cylinder (see \cite{FarKra}*{p. 193}).

\begin{definition}[\cite{Val2012}]
\begin{upshape}
A translation surface $S$ is said to be \emph{tame} if for every point
$x\in\widehat{S}$ there exists a neighborhood $U_{x}\subset\widehat{S}$
isometric to one of the following:

	\begin{itemize}
		\item[\textbf{(1)}] an open subset of the complex plane;
		
		\item[\textbf{(2)}] a neighborhood of the ramification point in a finite or infinite cyclic branched covering of the complex unit disk.
	\end{itemize}
    In the second case, if $U_{x}$ is isometric to a cyclic branched covering of finite degree $k_x+1$, then $x$ is called a \emph{finite cone-angle singularity} with cone angle
$\theta(x)=2\pi(k_x+1)$.
If $U_{x}$ is isometric to the infinite cyclic branched covering, then $x$ is called an \emph{infinite cone-angle singularity}.
    \end{upshape}
\end{definition}

We denote by
${\rm Sing}(\widehat{S})$
the set of all finite and infinite cone-angle singularities of $\widehat S$. Elements of ${\rm Sing}(\widehat{S})$ will simply be called \emph{cone singularities} or \emph{cone points}.

\begin{definition}
\begin{upshape}
Let $(S_{1},\mathcal{A})$ and $(S_{2},\mathcal{B})$ be tame translation surfaces. A homeomorphism $T: S_{1}\to S_{2}$ is called an \emph{isomorphism} if the following conditions hold:
\begin{itemize}
    \item[\textbf{(1)}] $T$ maps singularities to singularities and preserves their cone angles;
    
    \item[\textbf{(2)}] the push-forward of the translation structure $\mathcal{A}$ by $T$ coincides with $\mathcal{B}$. Equivalently, for every chart $(U,\varphi)\in\mathcal A$, the chart $(T(U),\varphi\circ T^{-1})$ belongs to the atlas $\mathcal{B}$.
\end{itemize}
Two tame translation surfaces $S_{1}$ and $S_{2}$ are said to be \emph{isomorphic} if there exists an isomorphism between them.
\end{upshape}
\end{definition}

\subsection{Slit constructions}

\begin{definition}[\cites{RamVal , wright-survey}]
\begin{upshape}
Let $m_1$ and $m_2$ be parallel markings of equal length on tame translation surfaces $S_1$ and $S_2$, respectively. Cutting $S_1$ and $S_2$ along $m_1$ and $m_2$ yields translation surfaces with boundary, denoted by $\widetilde{S}_1$ and $\widetilde{S}_2$. Each cut produces two boundary segments, so that the boundaries $\partial\widetilde{S}_1$ and $\partial\widetilde{S}_2$ each contain two copies of the corresponding marking.
Consider the disjoint union of $\tilde{S}_1$ and $\tilde{S}_2$. We can identify the boundary segments arising from $m_1$ with those arising from $m_2$ via translations, making sure to glue ``opposite'' sides. That is, the top of $m_1$ should be glued to the bottom of $m_2$ and vice versa. In doing so, one obtains a connected tame translation surface $S$ (see Figure \ref{Figure_gluing_marks}-a). 
    \begin{figure}[!ht]
		\begin{tabular}{ccc}
			\begin{tikzpicture}[baseline=(current bounding box.north), scale=0.8]
			\begin{scope}
			\clip (0,-0.2) rectangle (8,5.2);
	%%%%%%%%%%%%%%%%%%%%%%%%%%%%%%%%%%%%%%%%%%%%%%%
	\draw [line width=1pt, dashed]  (2,0) -- (6,0);
	\draw [line width=1pt, dashed]  (2,2) -- (6,2);
    \draw [line width=1pt, dashed]  (2,0) -- (2,2);
    \draw [line width=1pt, dashed]  (6,2) -- (6,0);
	\draw [line width=1pt, dashed]  (2,3) -- (6,3);
	\draw [line width=1pt, dashed]  (2,5) -- (6,5);
    \draw [line width=1pt, dashed]  (2,3) -- (2,5);
    \draw [line width=1pt, dashed]  (6,3) -- (6,5);
		        %%%%%%%%%%%%%%%%%%%%%%%%%%%%%%%%%%%%%%
     %\draw [green!30, fill=green!30, line width=1pt] (3.2,1.2) arc
	%[
	%start angle=45,
	%end angle=315,
	%x radius=3mm,
	%y radius =3mm
	%] ;
    %\draw [green!30, fill=green!30, line width=1pt] (4.8,0.8) arc
	%[
	%start angle=-135,
	%end angle=135,
	%x radius=3mm,
	%y radius =3mm
	%] ;
    %\draw [green!30, fill=green!30, line width=1pt] (3.2,4.2) arc
	%[
	%start angle=45,
	%end angle=315,
	%x radius=3mm,
	%y radius =3mm
	%];  
    %\draw [green!30, fill=green!30, line width=1pt] (4.8,3.8) arc
	%[
	%start angle=-135,
	%end angle=135,
	%x radius=3mm,
	%y radius =3mm
	%] ;%%%%%%%%%%%%%%%%%%%%%%%%%%%%%%%%%%%%%%%%%%%%%%%%%%555           
	\draw [color=red, line width=1pt] plot[smooth] coordinates {(3,1)(4,0.8)(5,1)};
	\draw [color=blue, line width=1pt] plot[smooth] coordinates {(5,1)(4,1.2)(3,1)};
	%%%%%%%%%%%%%%%%%%%%%%%%%%%%%%%%%%%%%%%%%%%%
	\draw [color=blue, line width=1pt] plot[smooth] coordinates {(3,4)(4,3.8)(5,4)};
	\draw [color=red, line width=1pt] plot[smooth] coordinates {(5,4)(4,4.2)(3,4)};
	%%%%%%%%%%%%%%%%%%%%%%%%%%%%%%%%
	\draw [purple, line width=1pt, fill=purple] (3,1) circle (0.1);
    \draw [line width=1pt, orange, fill=orange] (5,1) circle (0.1);
    \draw [purple, fill=purple, line width=1pt] (3,4) circle (0.1);
    \draw [line width=1pt, orange, fill=orange] (5,4) circle (0.1);
    %%%%%%%%%%%%%%%%%%%%%%%%%%%%%%%%%%%%%%%
	%%%%%%%%%%%%%%%%%%%%%%%%%%%%%
	%%%%%%%%%%%%%%%%%%%%%%%%%%%%%%%%%%%%%%%%%
	\node[red] at (4,0.4) {$\mathfrak{b}$};
	\node[blue] at (4,1.6) {$\mathfrak{a}$};
	\node[blue] at (4,3.4) {$\mathfrak{a}$};
	\node[red] at (4,4.6) {$\mathfrak{b}$};
	\end{scope}
			\end{tikzpicture}&&
            \begin{tikzpicture}[baseline=(current bounding box.north), scale=1]
			\begin{scope}
			\clip (-0.3,-1.4) rectangle (5.2,3.4);
    %\draw [green!30, fill=green!30, line width=1pt] (1.2,0.05) arc
	%[
	%start angle=15,
	%end angle=355,
	%x radius=2mm,
	%y radius =2mm
	%] ;
    \draw [->, >=latex, color=red, line width=1pt] plot[smooth] coordinates {(1,0)(4,0.2)(5,0.1)};
    \draw [->, >=latex, color=blue, line width=1pt] plot[smooth] coordinates {(1,0)(4,-0.2)(5,-0.1)};

    %\draw [green!30, fill=green!30, line width=1pt] (1.2,1.05) arc
	%[
	%start angle=15,
	%end angle=355,
	%x radius=2mm,
	%y radius =2mm
	%] ;
    \draw [->, >=latex, color=blue, line width=1pt] plot[smooth] coordinates {(1,1)(4,1.2)(5,1.1)};
    \draw [->, >=latex, color=black, line width=1pt] plot[smooth] coordinates {(1,1)(4,0.8)(5,0.9)};

    %\draw [green!30, fill=green!30, line width=1pt] (1.2,2.05) arc
	%[
	%start angle=15,
	%end angle=355,
	%x radius=2mm,
	%y radius =2mm
	%] ;
    \draw [->, >=latex, color=black, line width=1pt] plot[smooth] coordinates {(1,2)(4,2.2)(5,2.1)};
    \draw [->, >=latex, color=brown, line width=1pt] plot[smooth] coordinates {(1,2)(4,1.8)(5,1.9)};
    \draw [black, line width=1pt, fill=black] (1,0) circle (0.07);
    \draw [black, line width=1pt, fill=black] (1,1) circle (0.07);
    \draw [black, line width=1pt, fill=black] (1,2) circle (0.07);
    \node at (3,2.8) {$\vdots$};
    \node at (3,-0.6) {$\vdots$};
    \draw [line width=1pt, dashed]  (-0.2,-1.2) -- (-0.2,3.3);
    \draw [line width=1pt, dashed]  (0,-1.2) -- (5,-1.2);
    \draw [line width=1pt, dashed]  (5,-1.2) -- (5,-0.1);
    \draw [line width=1pt, dashed]  (5,0.2) -- (5,0.9);
    \draw [line width=1pt, dashed]  (5,0.2) -- (5,0.8);
    \draw [line width=1pt, dashed]  (5,1.1) -- (5,1.9);
    \draw [line width=1pt, dashed]  (5,2.1) -- (5,3.3);
    \draw [line width=1pt, dashed]  (0,3.3) -- (5,3.3);
    
    \node[red] at (2,0.5) {$c$};
    \node[blue] at (4.6,-0.4) {$b$};
    \node[blue] at (2,1.5) {$b$};
    \node[black] at (4.6,0.6) {$a$};
    \node[black] at (2,2.5) {$a$};
    \node[brown] at (4.6,1.6) {$d$};

    \node[black] at (0.5,-0.4) {$\gamma_{-1}(0)$};
    \node[black] at (0.5,0.6) {$\gamma_{0}(0)$};
    \node[black] at (0.5,1.6) {$\gamma_{1}(0)$};
    
            \end{scope}
            \end{tikzpicture}\\
            a. Gluing marks.&& b. Gluing rays.\\
	\end{tabular}	
		\caption{\emph{Operation of gluing.}}
		\label{Figure_gluing_marks}
	\end{figure}
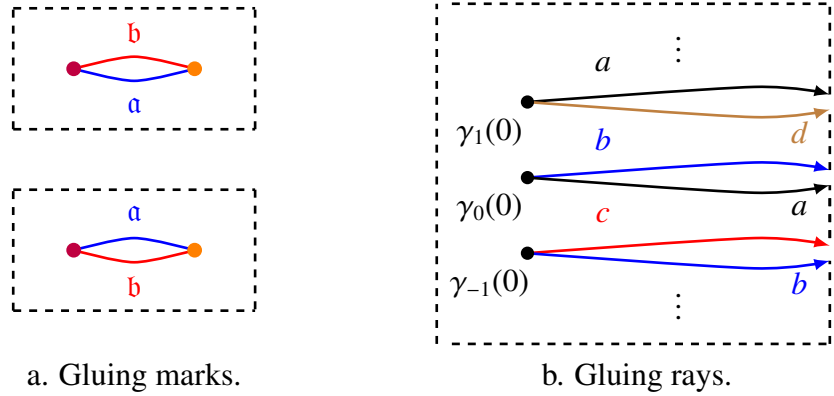
    
We denote this identification by
\[
m_{1}\sim_{\mathrm{glue}} m_2
\]
and refer to it as the \emph{gluing of the markings $m_1$ and $m_2$}. The resulting surface is written as
\[
S=(S_1\sqcup S_2)/m_{1}\sim_{\mathrm{glue}}m_{2}.
\]
We say that $S$ is obtained from $S_1$ and $S_2$ by \emph{gluing along $m_1$ and $m_2$}. More generally, the same construction can be performed simultaneously along any finite or countable collection of pairwise parallel markings of equal length.
    
A \emph{ray} on a tame translation surface $S$ is a geodesic $\gamma:[0,\infty)\to S$. The gluing operation extends naturally to finite or countable collections of pairwise parallel rays. In the countable case, let $\{\gamma_n: n\in\mathbb{Z}\}$ be a family of parallel rays on $S$. To ensure that the resulting quotient surface remains tame, we require the set of initial points
$\{\gamma_n(0):n \in\mathbb{Z}\}$
to be discrete in $S$. See Figure \ref{Figure_gluing_marks}-b.
    \end{upshape}
\end{definition}

\subsubsection{Strata of tame translation surfaces}

Let $x\in{\rm Sing}(\widehat{S})$. If $x$ is a finite cone-angle singularity with cone angle
$\theta(x)=2\pi(k_x+1)$,
then
\[
\delta(x)=\theta(x)-2\pi=2\pi k_x
\]
is called the \emph{excess angle} at $x$. In contrast, if $x$ is a infinite cone-angle singularity, its excess angle is $\delta(x)=\infty$.

If $S$ is a compact translation surface, then ${\rm Sing}(S)$ is finite. The Gauss--Bonnet theorem relates the excess angles of the singularities to the Euler characteristic of $S$.

\begin{theorem}[Gauss-Bonnet \cite{Schwartz}*{§17}, \cite{Ath-Ma}*{p.20}]\label{Theorem:Gauss_Bonnet}
	Let $S$ be an compact translation surface $S$, possibly with boundary. Then
	\[
	-\sum\limits_{x\in {\rm Sing}(S)}\delta(x) + \Delta = 2\pi \chi (S).
	\]
	Here $\chi(S)$ is the Euler characteristic of $S$ and $\Delta$ is  the sum of the turning 	angles at the corners of $\partial S$.
\end{theorem} 

When $S$ is a compact surfaces without boundary having genus $g\geq 2$, Gauss--Bonnet's theorem implies:

\begin{theorem}[\cite{Ath-Ma}*{Theorem 2.1.3, p.16}]
    Let $S$ be a compact translation surfaces of genus $g\geq 2$ with singularities $x_{1},\ldots,x_{l}$, for some $l\in\mathbb{N}$, where $x_{i}$ has cone angle $2\pi(k_{i}+1)$, such that $i\in\{1,\ldots,l\}$ and $k_{i}\in\mathbb{N}$. Then
    \begin{equation}\label{eq:excess_euler}
    k_{1}+\ldots+k_{l}=2g-2.
    \end{equation}
\end{theorem}

Fix a closed compact topological surface $S_g$ of genus $g\geq 2$, and let $\mathbf{k}=(k_{1},\ldots, k_{l})$ be a partition of $2g-2$, i.e., a sequence of positive integers satisfying \eqref{eq:excess_euler}. Then $\mathcal{H}(k_{1},\ldots, k_{l})$ is defined as the set of isomorphism classes of translation surfaces $(S_g,\mathcal{A})$ whose cone singularities have angles
\[
2\pi(k_{1}+1),\ldots,2\pi(k_{l}+1).
\]

${\rm Sing}(S)=\{x_{1},\ldots,x_{l}\}$ such that $\theta(x_{i})=2\pi(k_{i}+1)$, for each $i\in\{1,\ldots,l\}$. Our notation is symmetric, meaning $\mathcal{H}(k_{1},\ldots, k_{l}) = \mathcal{H}(k_{\pi(1)},\ldots, k_{\pi(l)})=\mathcal{H}(\mathbf{k})$ for any permutation $\pi \in {\rm S}_{l}$. The collection $\mathcal{H}(\mathbf{k})$ is called a \emph{stratum}.

One has then
\[
\mathcal{H}_{g}=\bigsqcup\limits_{\substack{l,\,(k_{1},\ldots, k_{l})\\ k_{1}\leq\ldots\leq k_{l}\\k_{1}+\ldots+k_{l}=2g-2}}\mathcal{H}(k_{1},\ldots, k_{l}).
\]
 The space $\mathcal{H}_{g}$ is called \emph{moduli space of pairs $(S_{g},\mathcal{A})$}. It is a complex algebraic  orbifold of dimension $4g-3$. For more details see \cite{Kont-Zor}.
 
In the non-compact setting, we introduce the following analogue of the classical stratum.

\begin{definition}
\begin{upshape}
Let $(S,\mathcal{A})$ be a tame translation surface, for each $n\in\mathbb{N}$, let ${\rm Sing}_{n}(S)$ be the set of all singularities of cone angle $2\pi (n+1)$ of $S$. Let ${\rm Sing}_{0}(S)$ denote the set of all singularities of infinite cone angle of $S$. The set of all finite and infinite cone-angle singularities of $\widehat{S}$ is written in the form
\[
{\rm Sing}(\widehat{S})=\bigsqcup\limits_{n=0}^{\infty} {\rm Sing}_{n}(S).
\]
For each $n\in\mathbb{N}_{0}$, set 
${\rm card}({\rm Sing}_{n}(S)):=k_{n}$,
where $0\leq k_{n}\leq \aleph_{0}$. The sequence 
\[
\mathbf{k}_{\infty}:=(k_{n})_{n\in\mathbb{N}_{0}}
\]
is called the \emph{singularity data} of $(S,\mathcal{A})$.

\end{upshape}
\end{definition}

\begin{remark}
The following observations are immediate.
\begin{enumerate}
\item[\textbf{(1)}] Let $J_{n}$ be a copy of the compact countable space
$\mathbb {N}_{0}\cup\{\infty\}$,
endowed with the topology inherited from the extended complex plane $\widehat{\mathbb{C}}$. Then the singularity data
$\mathbf{k}_{\infty}=(k_n)_{n\in\mathbb N_0}$
may be regarded as a point of the product space
$\prod\limits_{n\in\mathbb{N}_{0}}J_{n}$.
Since $\mathbb{N}_{0}\cup\{\infty\}$ is homeomorphic to the ordinal space $\omega+1$, each factor $J_{n}$ is homeomorphic to $\omega+1$.

\item[\textbf{(2)}] Isomorphic tame translation surfaces have identical singularity data.

\item[\textbf{(3)}] If $k_{n}=0$ for all $n\in\mathbb{N}_{0}$, then the corresponding translation surface has no cone singularities. Consequently, the only non-compact tame translation surfaces with trivial singularity data are the complex plane and the flat cylinder.
 \end{enumerate}   
\end{remark}

\begin{definition}
\begin{upshape}
Let $S$ be a non-compact surface and let $\mathbf{k}_{\infty}=(k_{n})_{n\in\mathbb{N}_{0}}
\in
\prod\limits_{n\in\mathbb{N}_{0}}J_{n}$. The \emph{non-compact stratum} with the singularity data $\mathbf{k}_{\infty}$ is the set
\[
\mathcal{H}(S,\mathbf{k}_{\infty})
\]
consisting of all isomorphism classes of tame translation structures $(S,\mathcal{A})$ whose singularity data is equal to $\mathbf{k}_{\infty}$.

Equivalently, $[(S,\mathcal{A})]
\in
\mathcal{H}(S,\mathbf{k}_{\infty})$ if and only if ${\rm card}({\rm Sing}_{n}(S))=k_{n}$ for every $n\in\mathbb{N}_{0}$.  

One has then
\[
\mathcal{H}(S)=\bigsqcup\limits_{\mathbf{k}_{\infty}\in \prod\limits_{n=0}^{\infty}J_{n}}\mathcal{H}(S,\mathbf{k}_{\infty}).
\]
The set $\mathcal{H}(S)$ is called \emph{moduli set of pairs $[(S,\mathcal{A})]$}.
\end{upshape}
\end{definition}

\section{Proof of \textbf{Theorem A}}\label{sec:proof_thma}

We begin by establishing several auxiliary lemmas that will be used in the proof of the main theorem.

\begin{lemma}\label{prop:polygonal-reps}
Let $S$ be a tame translation surface without infinite cone-angle singularities. Let $x, y \in S$ and let $[\sigma]$ be a homotopy class of arcs joining $x$ to $y$. Then $[\sigma]$ contains a polygonal representative.
\end{lemma}

\begin{proof}
Let $\sigma : [0,1]\to S$ be a representative of $[\sigma]$. 

\emph{Case 1.} Assume that $\sigma$ avoids $\text{Sing}(S)$ except possibly at its endpoints.

Since $\sigma([0,1])$ is compact, there exists a finite collection of tame translation charts
\[
(U_0,\varphi_0),\ldots,(U_m,\varphi_m)
\]
covering $\sigma([0,1])$. By refining this cover if necessary, we may choose a partition
\[
0=t_0<t_1<\cdots<t_{m+1}=1
\]
such that $\sigma([t_i,t_{i+1}])\subset U_i$ for each $i \in \{0, \ldots, m\}$ and $\sigma(t_i)\in U_{i-1}\cap U_{i}$ for each $i \in\{1,\ldots,m\}$.

For each $i$, the chart $\varphi_{i}:U_{i}\to V_{i}$ is an isometry onto a simply connected open subset (either of the complex plane $\mathbb{C}$ or $S'$ the ramification point in a finite cyclic branched covering of the complex unit disk). Hence, the segment in $\mathbb{C}$ (or $S$', respectively) joining $\varphi_i(\sigma(t_i))$ to $\varphi_i(\sigma(t_{i+1}))$ is well defined and lies in $V_i$. Pulling it back by $\varphi_i^{-1}$ yields a geodesic segment $\gamma_i \subset U_i$ joining $\sigma(t_i)$ to $\sigma(t_{i+1})$. In particular, each $\gamma_i$ is a polygonal arc and is homotopic to $\sigma|_{[t_i,t_{i+1}]}$ relative endpoints inside $U_i$.

We define
\[
\gamma := \gamma_0 * \gamma_1 * \cdots * \gamma_{m} ,
\]
where $*$ denotes concatenation of paths. Then $\gamma$ is a polygonal curve from $x$ to $y$.

Since each $\gamma_i$ is homotopic to $\sigma|_{[t_i,t_{i+1}]}$ relative endpoints, the concatenation $\gamma$ is homotopic to $\sigma$ relative endpoints. Therefore $[\sigma]$ contains a polygonal representative.

\emph{Case 2.} Assume now that the curve $\sigma$ passes through singular points $x_{1},\ldots,x_{\ell}$ of $S$. Then there exist parameters
\[
0=t_0 < t_1 < \cdots < t_\ell < t_{\ell+1}=1
\]
such that $\sigma(t_j)=x_j$ for each $j\in\{1,\ldots,\ell\}$. Set $\sigma_j := \sigma|_{[t_j,t_{j+1}]}$. Each subarc $\sigma_j$ is a curve with endpoints in (possibly singular) points of $S$, but whose interior contains no singularities. Hence, by the previous case, for each $j\in\{0,\ldots,\ell\}$ there exists a polygonal curve $\gamma_j$ homotopic to $\sigma_j$ relative endpoints.

We define the concatenation
\[
\gamma := \gamma_0 * \cdots * \gamma_\ell.
\]
Then $\gamma$ is a polygonal curve with endpoints $\sigma(0)$ and $\sigma(1)$. Moreover, since homotopies relative endpoints are preserved under concatenation, it follows that $\gamma$ is homotopic to $\sigma$ relative endpoints.
\end{proof}

\begin{corollary}\label{prop:polygonal-closed}
Let $S$ be a tame translation surface without infinite cone-angle singularities. Every homotopy class of simple closed curve of $S$ contains a simple closed polygonal representative.
\end{corollary}

The next result provides polygonal representatives for a pants decomposition of the Loch Ness monster.

\begin{lemma}\label{lemma:pants_decomposition}
    Let $S$ be a tame translation structure the Loch Ness monster without singularities of infinite cone angle. Then there exists a pants decomposition $P = \{\gamma_n\}_{n \in N}$ of $S$ where the $\gamma_n$ are pairwise disjoint essential polygonal simple closed curves.
\end{lemma}

\begin{proof}
    First we will show that there is a pants decomposition $P= \{\gamma_n\}_{n\in \N}$ of $S$ where for each $n \in \N$, there exists an $\epsilon_n>0$ such that the neighborhood $\mathcal N_{\epsilon_n}(\gamma_n)$ is disjoint from the other pants curves $P\setminus \{\gamma_n\}$.

    Indeed, since $S$ is one-ended, we can choose a nested exhaustion of $S$ by compact sets $K_1 \subseteq K_2 \subseteq K_3 \subseteq \cdots$ such that $K_1 \subseteq \text{int}\:K_2 \subseteq \text{int}\:K_3 \subseteq\cdots$, and such that each $K_m$ is a compact, connected subsurface of finite type with $U_m =S\setminus K_m$ connected. Given that $\sing(S)$ is discrete, we can choose these compact sets $K_m$ in such a way that for each $m \in \N$, the boundary $\partial K_m$ is a multicurve in $S$ avoiding $\sing(S)$. This guarantees that $A_m:= K_{m+1}-\text{int}\:(K_m)$ does not contain any cone points on its boundary.

    Let $\ell \in \{1,2, 3,\ldots\}$ be the minimum index for which the subsurface $K_\ell$ has negative Euler characteristic. As $K_\ell$ is a surface of finite type with negative Euler characteristic, it has a pants decomposition relative to its boundary (i.e., including its boundary curves) which avoids all cone points.

    Now suppose that we have a pants decomposition $P_m = \{\gamma_n\}_{n=1}^{n_m}$ of $K_m$ relative to its boundary such that all curves in the decomposition avoid cone points. Consider $A_n = K_{m+1}-\text{int}\:K_m$.
    \begin{itemize}
        \item If $A_m$ is a finite union of cylinders, we do nothing. The boundary curves in $\partial K_m$ and $\partial K_{m+1}$ are homotopic, and so in this step we add no new curves to our pants decomposition of $S$.

        \item Otherwise, at least one component of $A_m$ has negative Euler characteristic, and so we choose a pants decomposition of $A_m$ relative to the boundary avoiding all cone points. Add these curves to our pants decomposition $\{\gamma_n\}$.
    \end{itemize}
    By induction, we get a pants decomposition $P = \cup_{m=\ell}^\infty P_m$ on all of $S$. 

    Consider a curve $\gamma_n$ in the pants decomposition $P$ of $S$. By construction $\gamma_n$ does not contain any cone points, and it is on the boundary of at most two pairs of pants, $Y_1$ and $Y_2$, where perhaps $Y_1=Y_2$. Let $Q$ be the finite collection of boundary curves of $Y_1$ and $Y_2$, excluding the curve $\gamma_n$, and let
    \[
    \epsilon_n = \frac12 \min_{x\in \gamma_n}d(x, Q).
    \]
    By construction the $\epsilon_n$-neighborhood $\mathcal N_{\epsilon_n}(\gamma_n)$ is disjoint from $Q$ and in general from $P\setminus \{\gamma_n\}$. In fact, for any two $n, m \in \{1, 2, 3, \ldots\}$ with $m \neq n$, we have $\mathcal N_{\epsilon_n}(\gamma_n) \cap \mathcal N_{\epsilon_m}(\gamma_m) = \emptyset$. For each $n \in \{1, 2, 3, \ldots\}$, we can apply Lemma~\ref{prop:polygonal-closed} within the annulus $\mathcal N_{\epsilon_n}(\gamma_n)$ to find a polygonal representative in the free homotopy class $[\gamma_n]$ which remains in $\mathcal N_{\epsilon_n}(\gamma_n)$, proving the lemma.
\end{proof}

\begin{lemma}\label{lemma:collar}
Let $\gamma$ be an essential simple closed polygonal curve in a pants decomposition as described in Lemma~\ref{lemma:pants_decomposition}. Suppose that $\gamma$ is contained in an isometrically embedded Euclidean collar. Then there exists a translation chart $(U,\psi)$ defined on a neighborhood of $\gamma$ such that $\psi\circ \gamma$ is a polygonal curve in $\mathbb{C}$.
\end{lemma}

\begin{proof}
The compact subspace $\gamma([0,1])$ of $S$ there is covering for a finite number of charts 
\[
(U_{1},\varphi_{1}),\ldots,(U_{m},\varphi_{m})
\]
of the tame translation structure of $S$, for some $m\in\mathbb{N}$. We can assume the following facts:

\begin{itemize}
\item[\textbf{(1)}] the intersection $U_{i}\cap U_{i+1}\neq \emptyset$, for each $i\in\{1,\ldots,m-1\}$;
\item[\textbf{(2)}] $\gamma(0)\in U_{1}$ and $\gamma(1)\in U_{m}$;
\item[\textbf{(3)}] the union $\bigcup\limits_{i=1}^{m}U_{i}$ is a collar containing the image $\gamma([0,1])$,  without singular points in its interior and no closed polygonal curves of the pairs of pants descomposition of $S$ either. 
\end{itemize}

Recall that the charts $(U_{i},\varphi_{i})$ and $(U_{i+1},\varphi_{i+1})$ intersect the image $\gamma([0,1])$ and the transition map is locally in the form 
\[
\varphi_{i+1}\circ \varphi_{i}^{-1}(z)=z+c,
\]
for some constant $c\in \mathbb{C}$. This implies the following facts
\begin{itemize}
    \item[\textbf{(1)}] For each $x\in U_{i}\cap U_{i+1}$, the direction  of the curve polygonal curve $\gamma\circ \varphi_{i}$ and $\gamma\circ\varphi$ coincides, except in the corners.
    \item[\textbf{(2)}] Globally, the tangent vector is well-defined (where it exists) along the curve $\gamma$, meaning the tangent vector does not depend on the chart $(U_{i},\varphi_{i})$.
\end{itemize}

\textbf{Building the chart $(U,\psi)$.} By induction on $i\in\{1,\ldots,m\}$. For $i=1$, we take the chart $(U_{1},\psi_{i}=\varphi_{1})$. For $i=2$, given that the translation maps are restriction maps of translation, we can modify (by post-composing) the chart $(U_{2},\varphi_{2})$ by a suitable translation map $T_{2}:\mathbb{C}\to \mathbb{C}$ such that we obtain a new chart $(U_{2}, \psi_{2}=T_{2}\circ \varphi_{2})$, which is compatible with the tame translation structure of $S$ and  
\[
\psi_{1}(x)=\psi_{2}(x).
\] 
for each point $x\in U_{1}\cap U_{2}$. In particular, the above equality holds for each $x\in \gamma([0,1])\cap (U_{1}\cap U_{2})$. For each $2\leq i \leq m$, given the translation maps are restriction maps of translation, we can modify (by post-composing) the chart $(U_{i},\varphi_{i})$ by a suitable translation map $T_{i}:\mathbb{C}\to\mathbb{C}$ such that we obtain a new chart $(U_{i},\psi_{i}=T_{i}\circ \varphi_{i})$, which is compatible with the tame translation structure of $S$ and
\[
\psi_{i-1}(x)=\psi_{i}(x).
\] 
for each point $x\in U_{i-1}\cap U_{i}$. In particular, the above equality holds for each $p\in \gamma([0,1])\cap (U_{i-1}\cap U_{i})$. We then define the chart $(U,\psi)$ where 
\[
U=\bigcup\limits_{i=1}^{m}U_{i} \quad \text{ and } \quad
\psi|_{U_{i}}=\psi_{i}
\]
for each $i\in\{1,\ldots,m\}$. Note the Pasting Lemma guarantees $\psi$ is well-defined. Thus, the chart $(U,\psi)$ is on the tame translation structure of $S$ and the image $\psi\circ\gamma$ is a simple polygonal curve of $\mathbb{C}$.  
\end{proof}

\begin{lemma}
Let $\gamma$ be a closed polygonal curve and $\psi$ a chart as in Lemma~\ref{lemma:collar}, and let
$\tilde{\gamma}=\psi\circ\gamma$.
Denote by $\Delta$ the sum of the turning angles at the corners of the closed polygonal curve $\gamma$. Then
\[
\Delta=
\begin{cases}
2\pi, & \text{if }\tilde{\gamma}\text{ is a simple closed polygonal curve},\\
0, & \text{if }\tilde{\gamma}\text{ is a polygonal arc}.
\end{cases}
\]
\end{lemma}

\begin{proof}
    If $\tilde{\gamma}$ is a closed polygonal curve, then $\tilde{\gamma}$ is a Jordan curve as shown Figure \ref{Figure_closed_polygonal_curve}-a. Let $K$ the bounded connected component of $\mathbb{C}-\tilde{\gamma}([0,1])$. We remark that
    \begin{itemize}
        \item[\textbf{(1)}] the boundary of closure $\overline{K}$ is the Jordan curve;
        \item[\textbf{(2)}] the closure $\overline{K}$ is homeomorphic to a closed disk $\mathbb{D}$.
        \item[\textbf{(3)}] In the proof of Lemma~\ref{lemma:collar} the curve $\tilde \gamma$ is constructed using the translation structure, and so turning angle sum of $\gamma$ is equal to the turning angle sum of $\tilde \gamma$.
    \end{itemize}

    The Euler characteristic of a closed disk is $\chi(\mathbb{D})=1$. Using the Gauss-Bonnet Theorem \ref{Theorem:Gauss_Bonnet}
    \[
    \Delta=2\pi\chi(\overline{K})=2\pi\chi(\mathbb{D})=2\pi.
    \]

   \begin{figure}[!ht]
    \begin{tabular}{ccc}
	\begin{tikzpicture}[baseline=(current bounding box.north), scale=1]
	\begin{scope}
	\clip (-2,0) rectangle (5,4);
    \draw [black!30, line width=1pt, ->, >=latex]  (0,2) -- (5,2);
    \draw [black!30, line width=1pt, ->, >=latex]  (1,-1) -- (1,4);
     \draw[thick, blue]
        (0.5,0.8) --
        (1.2,2.3) --
        (2.0,3.2) --
        (3.4,3.0) --
        (4.3,1.8) --
        (3.2,0.6) --
        (1.5,0.3) --
        cycle;
    \node at (1.9,1.1) {$\widetilde{\gamma}(0)=\widetilde{\gamma}(1)$};
    \node at (0,3) {$\mathbb{C}$};
    \draw [line width=1pt, fill=black] (0.5,0.8) circle (0.09);
    
\end{scope}
\end{tikzpicture}
&&
\begin{tikzpicture}[baseline=(current bounding box.north), scale=1]
	\begin{scope}
	\clip (-2,-0.5) rectangle (5,4);
    \draw [black!30, line width=1pt, ->, >=latex]  (-0.8,1) -- (3,1);
    \draw [black!30, line width=1pt, ->, >=latex]  (1,-1) -- (1,4);
     \draw[thick, blue]   
        (0,0) -- 
        (-0.5,0.5) -- 
        (0.5,1.5) -- 
        (-0.5,2.5)--
        (0,3);
        \draw[thick, red]
        (0,3)--
        (0,3.3)--
        (2,3.3)--
        (2,-0.3)--
        (0,-0.3)--
        (0,0);
        \draw [red, dashed, line width=1pt]  (0,-0.3) -- (0,3.3);
        \node at (-0.6,0) {$\widetilde{\gamma}(0)$};
        \node at (3,2.5) {$\mathbb{C}$};
    \draw [line width=1pt, fill=black] (0,0) circle (0.09);
    \node at (-0.6,3) {$\widetilde{\gamma}(1)$};
    \draw [line width=1pt, fill=black] (0,3) circle (0.09);
        \end{scope}
\end{tikzpicture}\\
&&\\
        a. $\tilde{\gamma}$ a Jordan curve. && b. Case $\tilde{\gamma}(0)\neq\tilde{\gamma}(1)$.
    \end{tabular}
    \caption{\emph{Closed polygonal curve.}}
		\label{Figure_closed_polygonal_curve}
    \end{figure}
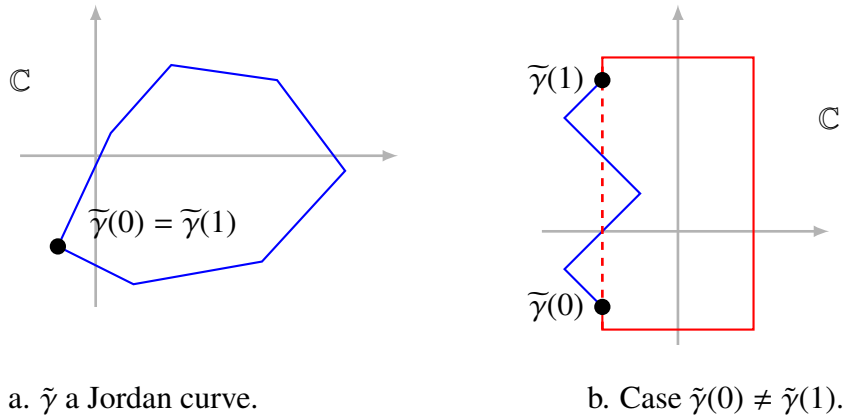

    In the other case, $\tilde{\gamma}(0)\neq \tilde{\gamma}(1)$. On $\mathbb{C}$ we can draw a rectangle $R$ such that one of its sides (without the corners) contains the ends points $\tilde{\gamma}(0)$ and $\tilde{\gamma}(1)$, and the other three sides do not intersect the image $\tilde{\gamma}([0,1])$. Let $\Gamma$ the polygon obtain by removing the straight line segment with end points $\tilde{\gamma}(0)$ and $\tilde{\gamma}(1)$ of $R$ and pasting the polygonal curve $\tilde{\gamma}([0,1])$. See Figure \ref{Figure_closed_polygonal_curve}-b.

    Let $\tilde \Delta$ be the sum of the turning angles at the corners of $\tilde\gamma$, including the turning angles between $\tilde\gamma(0)$ and the rectangular curve and $\tilde\gamma(1)$ and the rectangular curve. By construction of the rectangular curve, we guarantee that $\tilde\Delta = \Delta$.
    
    Let $K$ be region inside of $\Gamma$, the closure $\overline{K}$ is homeomorphic to a closed disk. Using the Gauss-Bonnet Theorem \ref{Theorem:Gauss_Bonnet} 
\[
    \tilde\Delta+\frac{\pi}{2}+\frac{\pi}{2}+\frac{\pi}{2}+\frac{\pi}{2}=2\pi\chi(\overline{K})=2\pi\chi(\mathbb{D})=2\pi.
    \]
    Thus, $\Delta = \tilde \Delta = 0$.
\end{proof}

\begin{proof}[Proof of \textbf{Theorem A}]
Suppose that there exists a Loch Ness monster $S$ having a tame translation structure with finitely many singularities of finite cone angle. There is a family $\{\gamma_{n}:n\in\mathbb{N}\}$ of disjoint simple closed curves in $S$ decomposing the Loch Ness monster into pair of pants $\{P_{k}:k\in\mathbb{N}\}$, where each $P_{k}$ is a pair of pants. Since the curves in $\{\gamma_n\}$ are pairwise disjoint, using Corollary \ref{prop:polygonal-closed} and Lemma \ref{lemma:pants_decomposition} we can find a polygonal representative $\gamma_n' \in [\gamma_n]$ in each free homotopy class such that the curves in $\{\gamma_n'\}$ are still pairwise disjoint. 

Since the polygonal curves $\{\gamma_n'\}$ are pairwise disjoint, each singular point of finite cone angle can be contained in at most two pairs of pants of the pants decomposition. Since we have finitely many singular points but infinitely many pairs of pants in the decomposition, there must be a pair of pants $P_\ell$ for some $\ell \in \N$ which contains no singular points in its interior or boundary curves.

We apply Lemma~\ref{lemma:collar}, then show that the signed angle sum of the boundary polygonal curves $\gamma_n$ must equal 0.

We recall the Gauss--Bonnet Theorem for a compact translation surface $\tilde{S}$, possible with boundary
\[
	-\sum\limits_{x\in {\rm Sing}(\tilde{S})}\delta(x) + \Delta = 2\pi \chi (\tilde{S}).
	\]

In our case since the pairs of pants $P_\ell$ is everywhere flat, the first sum is identically 0. Since the boundary curves of $P_\ell$ are polygonal (i.e., piecewise geodesic), $\Delta$ is equal to the sum of the signed angles by which the polygonal curves turn at the boundaries, and we know that for each of these curves this angle sum is equal to 0 or $2\pi$. Since $\chi(P_\ell) = -1$, we have that
\[
 0, 2\pi, 4\pi,\text{ or } 6\pi = -2\pi,
\]
a contradiction.
\end{proof}

\section{Proof of \textbf{Theorem B}}\label{sec:proof_thmb}

Let $\mathbf{k}_{\infty}=(k_{n})_{n\in\mathbb{N}_{0}}\in \prod\limits_{n\in\mathbb{N}_{0}}J_{n}$ be a sequence that is not ruled out by \textbf{Theorem A}. The following three existence theorems together imply our main \textbf{Theorem B}. In each case, there exists a tame
translation structure on the Loch Ness monster whose singularity data is the prescribed sequence
$\mathbf{k}_{\infty}$. As we observe in Remark~\ref{remark:LNM_strata_1_0_0_0} below, our constructions can be perturbed to build infinite-dimensional families in each of the strata represented by a vector $\kinf$ as below.

\begin{theorem}\label{theorem_case_1_MTB}
If $k_0\neq 0$ and $k_n=0$ for all $n\geq 1$, then there exists a tame translation structure on the Loch Ness monster whose singularity data is exactly $\kinf$.
\end{theorem}

\begin{theorem}\label{theorem_case_2_MTB}
If $k_0=0$ and $\sum\limits_{n=1}^\infty k_n = \infty$, then there exists a tame translation structure on the Loch Ness monster whose singularity data is exactly $\kinf$.
\end{theorem}

\begin{theorem}\label{theorem_case_3_MTB}
If $k_0 \neq 0$ and $k_n \neq 0$ for at least one $n \geq 1$, then there exists a tame translation structure on the Loch Ness monster whose singularity data is exactly $\kinf$.
\end{theorem}

For each of these cases, we construct a tame translation structure on the Loch Ness monster whose singularity data is precisely $\mathbf{k}_{\infty}$. We begin with several auxiliary constructions and lemmas that will be used throughout the proof.

\subsection{Auxiliary constructions and lemmas}

\begin{construction}\label{cons:torus_with_boundary}
\begin{upshape}
Let $U\subset \mathbb{C}$ be a connected open subset, and let $P\subset U$ be a square such that exactly one of its vertices belongs to the boundary $\partial U$ of $U$, as shown Figure \ref{fig:open-minus-square}.
\begin{figure}[!ht]
	\centering
    \begin{tabular}{ccc}
	\begin{tikzpicture}[baseline=(current bounding box.north)] 
	\begin{scope}[scale=0.6]
	\clip (-5,-2) rectangle (5,5);
	%%%%%%%%%%%%%%%%%%%%%%%%%%%%%%%%%%%%%%%%%%%%%%
	\draw[fill=green!20, dashed, line width=1.0pt] (-3.5,-0.5) to[out=15, in=180] (-1,-1.5) to[out=0, in=210] (2.5,-0.5) to[out=0, in=-90] (2,4) to[out=90, in=0] (-2,4) to[out=180, in=90] (-3.8,2) to[out=-90, in=180] cycle;
	\draw [white, fill=gray!20, line width=1.0pt] (-3,-0.5) --(-1,-0.5) -- (-1,1.5)-- (-3,1.5)-- cycle;
    \draw [red, line width=1.0pt] (-3,-0.5) -- (-1,-0.5);
    \node[red] at (-2,-0.5){$||$};
    \draw [red, line width=1.0pt] (-3,1.5) -- (-1,1.5);
    \node[red] at (-2,1.5){$||$};
    \draw [blue, line width=1.0pt] (-1,-0.5) -- (-1,1.5);
    \node[blue] at (-1,0.5){$--$};
    \draw [blue, line width=1.0pt] (-3,-0.5) -- (-3,1.5);
    \node[blue] at (-3,0.5){$--$};
    \draw [line width=1pt, fill=black] (-3,-0.5) circle (0.1);
    \draw [line width=1pt, fill=black] (-1,-0.5) circle (0.1);
    \draw [line width=1pt, fill=black] (-1,1.5) circle (0.1);
    \draw [line width=1pt, fill=black] (-3,1.5) circle (0.1);
	\node at (1,3){$U$};
    \node at (0,1){$P$};
	\end{scope}
	\end{tikzpicture}
    &&		\begin{tikzpicture}[baseline=(current bounding box.north)]
    \begin{scope}[scale=0.6]
	\clip (-3,-3) rectangle (3,3);
						
	%%%%%%%%%%%%%%%%%%%%%%%%%%%%
	%%%%%%%%%
    \draw [white, fill=green!20, line width=1.0pt] (-2.5,-2.5) --(2.5,-2.5) -- (2.5,2.5)-- (-2.5,2.5)-- cycle;
	\draw [red, line width=1pt] (-2.5,-2.5) -- (2.5,-2.5);
    \node[red] at (0,-2.5){$||$};
    \node[red] at (0,2.5){$||$};
	\draw [red, line width=1pt] (-2.5,2.5) -- (2.5,2.5); \draw [blue, line width=1pt] (2.5,-2.5) -- (2.5,2.5);
    \draw [blue, line width=1pt] (-2.5,-2.5) -- (-2.5,2.5);
    \node[blue] at (-2.5,0){$--$};
    \node[blue] at (2.5,0){$--$};
    \draw[fill=gray!20, dashed, line width=1pt] (-2.5,-2.5) to[out=10, in=10] (0,0) to[out=90, in=90]  cycle;
    \draw [line width=1pt, fill=black] (-2.5,2.5) circle (0.1);
    \draw [line width=1pt, fill=black] (2.5,-2.5) circle (0.1);
    \draw [line width=1pt, fill=black] (2.5,2.5) circle (0.1);
    \draw [line width=1pt, fill=black] (-2.5,-2.5) circle (0.1);
	\end{scope}
		\end{tikzpicture}
			\end{tabular}
	\caption{\emph{Torus with a single puncture.}}
	\label{fig:open-minus-square}
\end{figure}
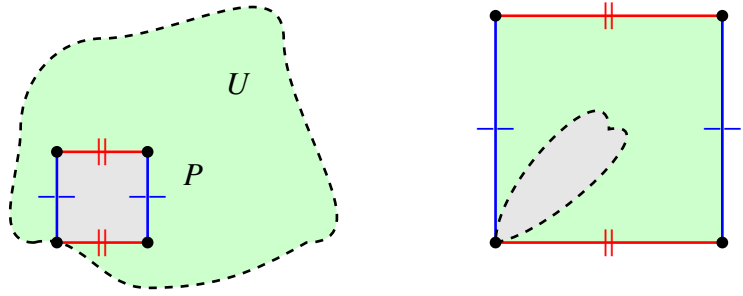
Remove the interior of $P$ and its vertices from $U$. Identifying each pair of opposite sides of $P$ by translation yields a translation surface homeomorphic to a torus with a single puncture.
\end{upshape}
\end{construction}

 \begin{lemma}\label{lemma:one_ends_glue_case_I}
Let $\gamma_{1}$ and $\gamma_{2}$ be parallel rays on translation surfaces $S_{1}$ and $S_{2}$, respectively. If both $S_{1}$ and $S_{2}$ have exactly one end, then the glued surface
\[
S:=(S_{1}\sqcup S_{2})/\gamma_{1}\sim_{\mathrm{glue}}\gamma_{2}
\]
also has exactly one end. 
%If the initial points $p_1$ and $p_2$ of the parallel rays have cone angles $2\pi\ell_1$ and $2\pi\ell_2$, respectively, for $\ell_1, \ell_2\in \N$, then after identification, the cone point $[\ell_1] = [\ell_2]$ has cone angle $2\pi(\ell_1+\ell_2)$.
\end{lemma}

\begin{proof}
Let $K$ be a compact subset of $S$. We must guarantee the existence of a compact subset $K'\subset S$ such that $K\subset K'$ and $S-K'$ is connected. The following surfaces with boundary are necessary to the proof.

Let $\pi:S_{1} \sqcup S_{2}\to S$ be the quotient map associated with the gluing construction. We note that the inverse image $\pi^{-1}(K)$ is a compact subset of $S_{1}\sqcup S_{2}$. This implies that $\pi(K)\cap S_{j}$ is a compact subset of $S_{j}$, for each $j\in\{1,2\}$. Since $S_{j}$ has only one end, there exists a compact subset $K'_{j}\subset S_{j}$ such that $\pi(K)\cap S_{j}\subset K'_{j}$ and $S_{j}-K'_{j}$ is connected. Enlarging $K_{j}'$ if necessary, we may assume that $K_{j}'$ contains the straight line segment $\gamma_{j}\left([0,\varepsilon)\right)$ of the infinite ray $\gamma_{j}\left([0,\infty]\right)$, for some positive real number $\varepsilon>0$. We remark that the infinite ray $\gamma_{j}((\varepsilon,\infty))$ belongs to the connected $S_{j}-K_{j}'$.

By construction, we hold the following facts:
\begin{itemize}
\item[\textbf{$\bullet$}] The complement $(S_{1}+S_{2})-(K_{1}'+K_{2}')$ is composed by the connected componnets $C_{1}:=S_{1}-K_{1}$ and $C_{2}:=S_{2}-K_{2}$.  

\item[\textbf{$\bullet$}] The union $K_{1}'\sqcup K_{2}'$ is a compact subset of $S_{1}\sqcup S_{2}$, which satisfies $\pi^{-1}(K)\subset K_{1}'\sqcup K_{2}'$. 

\item[\textbf{$\bullet$}] Since the projection map $\pi$ is continuous, the direct image $K:=\pi(K_{1}'+K_{2}')$ is a compact subset of $S$, such that $K\subset K'$.

\item[\textbf{$\bullet$}] The complement $S-K'$ is equal to quotient space obtained by restricting the gluing operation on the union of the connected components $S_{1}-K_{1}$ and $S_{2}-K_{2}$. Thus, it can be written in the form
\[
S-K'=\left(C_{1}+C_{2}\right)/\gamma_{1}\sim_{\rm glue}\gamma_{2}.
\]
\end{itemize}

Since the gluing operation restriction adequately identifies the rays $\gamma_{1}((\varepsilon,\infty))$ and $\gamma_{2}((\varepsilon, \infty))$, it does not create new connected components in the quotient space $S-K'$. Indeed, the gluing operation turns the disjoint union of the connected components into a connected quotient space. This proves that $S-K'$ is connected.
\end{proof}

\begin{lemma}
\label{lem:2n-gon}
Let $U\subset S$ be a connected open subset homeomorphic to a disk which does not contain any cone point singularities. Let $P\subset U$ be a regular $2n$-gon contained in $U$, and remove the interior of $P$. Identifying each pair of opposite sides of $P$ by translation yields a translation surface homeomorphic to a surface of genus $\lfloor n/2\rfloor$ with a single puncture. 
\begin{enumerate}[(a)]
    \item For $n$ even (that is, if $P$ is a $4m$-gon), then the resulting translation surface has one cone point of angle $(2n+2)\pi$.
    \item If $n$ is odd, then the resulting translation surface has two cone points of angle $(n+1)\pi$.
\end{enumerate} 
\end{lemma}

\begin{proof}
    Consider the regular $2n$-gon $P$, which has interior angle sum $(2n-2)\pi$. Let the {\em outer angle} at each vertex be the conjugate angle, that is, $2\pi - (\text{inner angle)}$. The sum of the outer angles, therefore, is 
    \[(2n)(2\pi) - (2n-2)\pi = (2n+2)\pi.\]
    This sum of the outer angles is sum of the cone angles of the equivalence classes of the vertices of $P$ after removing the interior of $P$ and identifying opposite edges.

    \begin{enumerate}[(a)]
    \item If $n$ is even and $P$ is a $4m$-gon, then there is only one equivalence class of vertices after gluing. Thus that vertex has cone angle equal to the total sum of the outer angles, $(2n+2)\pi$.
    \item If $n$ is odd, then there are two equivalence class of vertices after gluing. Each equivalence class contains $n$ vertices, and so each cone point has angle equal to half of the total sum of the outer angles, that is, $(n+1)\pi$.
\end{enumerate} 

The fact that the resulting surface is homeomorphic to a surface of genus $\lfloor n/2 \rfloor$ follows from the same inversion-type argument as shown in Figure~\ref{fig:open-minus-square}.
\end{proof}

\begin{lemma}
\label{lem:2n-gon-on-sing}
Let $q$ be a cone singularity of finite cone angle $2\pi k$, and let $U$ be a connected open set contained in the domain of a chart from $S \to \C$ which contains $q$ and no other singularities. Then define $P$ to be a regular $2n$-gon within $U$ so that one of its vertices is equal to $q$. By removing the interior of $P$ and identifying pairs of opposite sides, we increase the genus of the translation surface by $\lfloor n/2\rfloor$. 
\begin{enumerate}[(a)]
    \item If $n$ is even, the cone angle of $q$ increases by $2n\pi$ radians, and no other cone singularities are introduced.
    \item If $n$ is odd, then the cone angle of $q$ increases by $(n-1)\pi$, and a new singularity of cone angle $(n+1)\pi$ is introduced.
\end{enumerate} 
\end{lemma}

\begin{proof}
Let $p:M_k\to\mathbb{C}$ be the degree-$k$ cyclic covering of the complex plane $\mathbb{C}$ branched over the origin $\textbf{0}$. Then the singular point $q$ in $S$ has an open neighborhood isometric to an open neighborhood $U$ of $p\inv(\bf{0})$ in $M_k$, as shown in Figure~\ref{fig:remove_polygon}. Let $P$ be a regular $2n$-gon contained in the intersection of $U$ with the final lower halfplane on the rightmost side of Figure~\ref{fig:remove_polygon}, where exactly one vertex of $P$ intersects $p\inv(\bf{0})$. %Enumerate the edges of $P$ counterclockwise starting at $p\inv(\bf{0})$, starting with $s_1$ and ending with $s_{2n}$. Let $a$ be the measure of the angle between the horizontal segment $A_{2k-1}$ and the side $s_1$ of $P$, and let $c$ be the measure of the angle between the final side $s_{2n}$ of $P$ and the horizontal segment $A_{2k}$. Let $b$ be the internal angle of P at the vertex $p\inv(\bf{0})$. Then we have $a + b + c = \pi$.

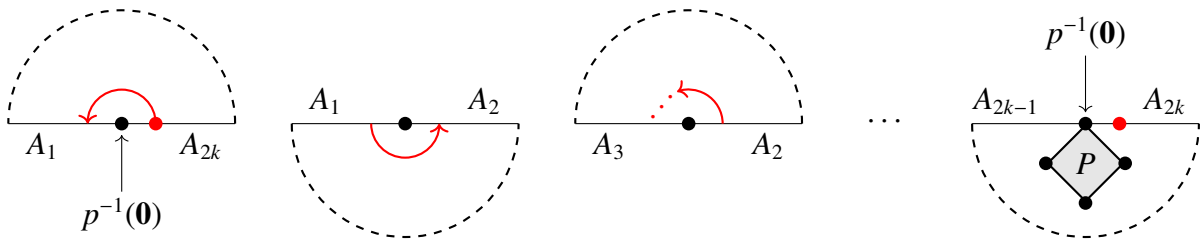
\begin{figure}[!ht]
    \begin{center}
	\begin{tikzpicture}[scale=1.5]

    \draw[dashed,line width=.8pt] (0,0) arc (0:180:1cm);
    \draw (-2,0) -- (0,0);
    \fill (-1,0) circle (0.06);
    \node at (-1.7, -.2) {$A_1$};
    \node at (-.3, -.2) {$A_{2k}$};
    \node at (-1, -.8) {$p\inv(\textbf{0})$};
    \draw[->] (-1,-.6) -- (-1,-.1);
    \node at (7.5, .8) {$p\inv(\textbf{0})$};
    \draw[->] (7.5,.6) -- (7.5,.1);
    \fill[red] (-.7,0) circle (0.06);
    \draw[line width=.8, red,->] (-.7,0) arc (0:180:.3cm);

    \draw[dashed,line width=.8pt] (.5,0) arc (180:360:1cm);
    \draw (.5,0) -- (2.5,0);
    \fill (1.5,0) circle (0.06);
    \node at (.8, .2) {$A_1$};
    \node at (2.2, .2) {$A_{2}$};
    \draw[line width=.8, red,->] (1.2,0) arc (180:360:.3cm);

    \draw[dashed, line width=.8pt] (5,0) arc (0:180:1cm);
    \draw (3,0) -- (5,0);
    \fill (4,0) circle (0.06);
    \node at (3.3, -.2) {$A_3$};
    \node at (4.7, -.2) {$A_2$};
    \draw[line width=.8, red,->] (4.3,0) arc (0:110:.3cm);
    \node[red, rotate = 50] at (3.8, .15) {$\cdots$};

    \node at (5.75, 0) {$\cdots$};

    \draw[dashed,line width=.8pt] (6.5,0) arc (180:360:1cm);
    \draw (6.5,0) -- (8.5,0);
    %\draw[line width=.8, red] (7.25,0) arc (180:225:.25cm);
    \node[draw, line width=.8pt, regular polygon, regular polygon sides=4, minimum size=1cm, fill=gray!20, rotate = 45] at (7.5, -.35) {};
    %\draw[line width=.8, blue] (7.75,0) arc (0:-135:.25cm);
    %\draw[line width=.8, orange] (7.75,0) arc (0:-45:.25cm);
    \fill (7.5,0) circle (0.06);
    \fill (7.5,-.7) circle (0.06);
    \fill (7.15,-.35) circle (0.06);
    \fill (7.85,-.35) circle (0.06);
    \node at (6.8, .2) {$A_{2k-1}$};
    \node at (8.2, .2) {$A_{2k}$};
    %\node[red] at (7.1, -.15) {$a$};
    \node at (7.5, -.35) {$P$};
    %\node[orange] at (7.9, -.15) {$c$};
    \fill[red] (7.8,0) circle (0.06);
    
	\end{tikzpicture}
	\end{center}
	\caption{\em Removing a regular polygon at a cone singularity.}
		\label{fig:remove_polygon}
\end{figure}

To compute the cone angle at $p\inv(\textbf{0})$, let us begin in the first upper halfplane on the segment labeled $A_{2k}$. Going counterclockwise, we traverse $2k-1$ halfplanes accumulating $(2k-1)\pi$ radians before we reach the final lower halfplane along the segment $A_{2k-1}$. In this final lower halfplane, the angle accumulated is related to the sum of the outer angles as described in the proof of Lemma~\ref{lem:2n-gon}, but in the context of the halfplane, the outer angle at the vertex $p\inv(\textbf{0})$ is reduced by $\pi$. Thus we conclude the following:
\begin{enumerate}[(a)]
    \item If $n$ is even, then the angle accumulated in the final lower halfplane is $(2n+1)\pi$, and the total cone angle at $p\inv(\textbf{0})$ is $(2k+2n)\pi$. Since all of the vertices of $P$ are identified after gluing, no additional singularities are introduced.
    \item If $n$ is odd, then the angle accumulated in the equivalence class of $p\inv(\textbf{0})$ in the lower halfplane is $n\pi$, and the total cone angle at $p\inv(\textbf{0})$ is $(2k + n -1)\pi$. A new singularity of cone angle $(n+1)\pi$ is also introduced.
\end{enumerate}

\end{proof}

\subsection{Proof of Theorem \ref{theorem_case_1_MTB}}\label{proof_theorem_Case_1_MTB} ($k_0 \neq 0$ and $\sum\limits_{n=1}^\infty k_n =0 $)

We start by constructing a tame translation structure on a surface $S$. Then we shall prove that $S$ is the Loch Ness monster. 

Let $\pi:\widetilde{S}\to\mathbb{C}$ be the infinite cyclic covering of the complex plane $\mathbb{C}$ branched over the origin $\textbf{0}$. The space $\widetilde{S}$ can be obtained as follows. For each $n\in\mathbb{Z}$, let $\mathbb{C}_{n}$ a copy of the complex plane $\mathbb{C}$. Then $\widetilde S$ is obtained by taking the union $\bigcup\limits_{n\in\mathbb{Z}}\mathbb{C}_{n}$, slitting the positive real ray (including the origin) of each $\C_n$, and gluing the top boundary of the slit in $\C_n$ with the bottom boundary of the slit in $\C_{n-1}$ for all $n \in \Z$. See Figure \ref{Figure_infinite_covering}. For each $n\in\mathbb{Z}$, the \emph{sheet} $L_{n}$ of $\widetilde S$ is the subsurface with boundary obtained by taking the copy $\mathbb{C}_n$ and cutting along the positive real axis. We then have the following facts:
    \begin{itemize}
    \item[\textbf{(1)}] Since the complex plane $\mathbb{C}$ is a translation surface with the global chart ${\rm Id}$, the open subset $\mathbb{C}^{\ast}=\mathbb{C}-\{\textbf{0}\}$ inherits this translation structure. 
    
    \item[\textbf{(2)}] By using this covering function $\pi$, we lift the translation structure of $\mathbb{C}^{\ast}$ on $\widetilde{S}-\{\pi^{-1}(\textbf{0})\}$. 
    
    \item[\textbf{(3)}] $\widetilde{S}-\{\pi^{-1}(\textbf{0})\}$ is a tame translation surface with only one singularity of infinite cone angle. 
    
    \item[\textbf{(4)}] The surface  $\widetilde{S}-\{\pi^{-1}(\textbf{0})\}$ has genus zero and only one end. 
    \end{itemize}
\vspace{2mm}

Let $N = k_0$. Define the complex numbers
    $z_{1}=\textbf{0}$ and $z_{\ell}=-2-i(\ell+1)$ for $2\leq \ell\leq N$.
    Now, on $\mathbb{C}$ we draw the unit square $P$ generated by the complex numbers $-1$ and $i$. Moreover, for each $2\leq \ell\leq N$, we take the ray $a_{\ell}$ in the left-halfplane having initial point $z_{\ell}$  parallel to the negative real axis as shown Figure \ref{Figure_infinite_covering}. 
    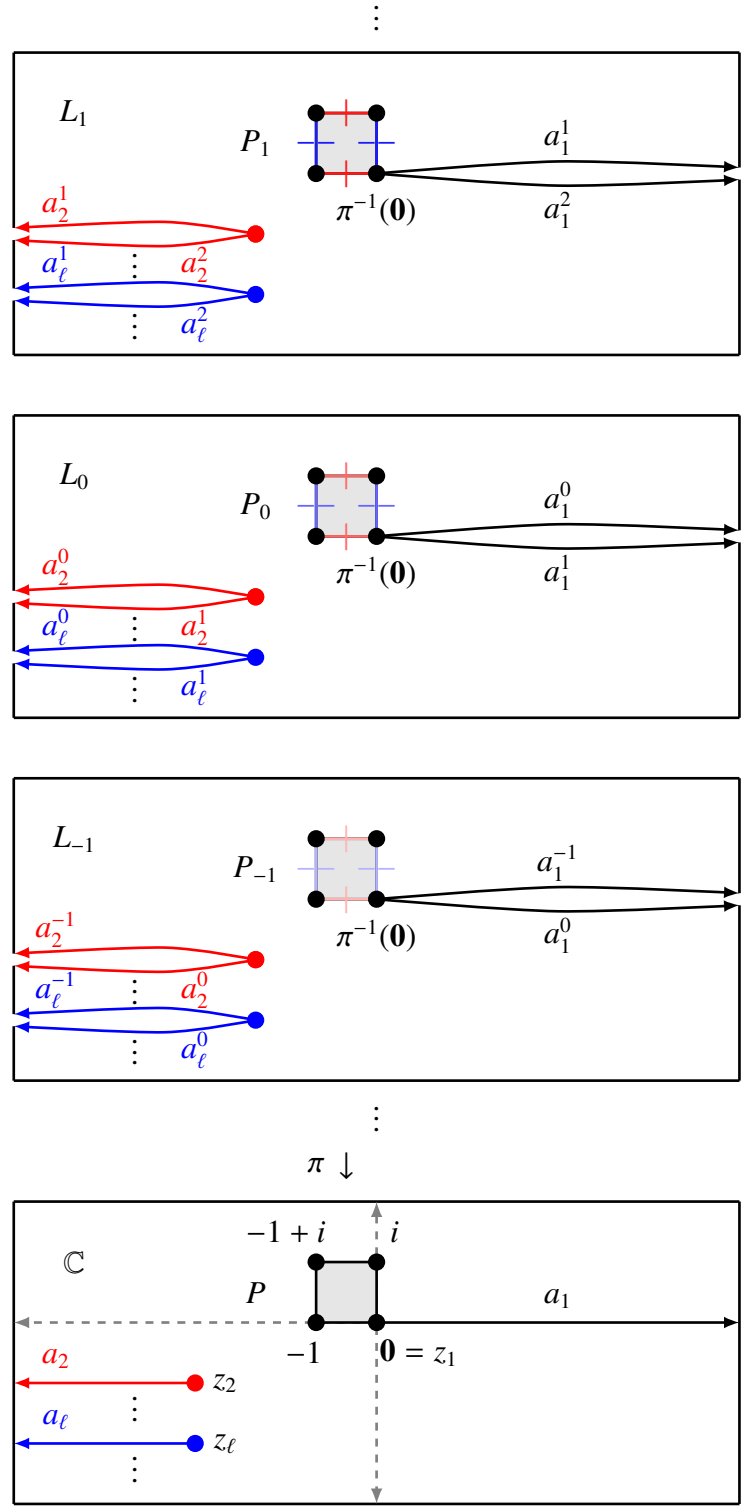
\begin{figure}[!ht]
		\begin{center}
			\begin{tikzpicture}[baseline=(current bounding box.north), scale=0.8]
				\begin{scope}
					\clip (-9.1,-8.1) rectangle (9.1,17);
	%%%%%%%%%%%%%%%%%%%%%%%%%%%%%%%%%%%%%%%%%%%%%%%
    %%%%%%%%%Complex-plane
	\draw [line width=1pt]  (-6,-8) -- (6,-8);
	\draw [line width=1pt]  (-6,-3) -- (6,-3);
    \draw [line width=1pt]  (-6,-3) -- (-6,-8);
    \draw [line width=1pt]  (6,-3) -- (6,-8);
    \draw [black!50, line width=1pt,  ->, >=latex, dashed] (0,-5) -- (-6,-5);
    \draw [black!50, line width=1pt,  <->, >=latex, dashed] (0,-3) -- (0,-8);                 
    \draw [fill=gray!20, line width=1.0pt] (0,-5) --(0,-4) -- (-1,-4)-- (-1,-5)-- cycle;
    \draw [line width=1pt,  ->, >=latex,] (0,-5) -- (6,-5);
    
    \draw [line width=1pt, red, fill=red] (-3,-6) circle (0.12);
    \draw [red, line width=1pt,  ->, >=latex,] (-3,-6) -- (-6,-6);
    \draw [line width=1pt, blue, fill=blue] (-3,-7) circle (0.12);
    \draw [blue, line width=1pt,  ->, >=latex,] (-3,-7) -- (-6,-7);
    \node at (-2.5,-6) {$z_{2}$};
    \node at (-2.5,-7) {$z_{\ell}$};
    \node at (-4,-6.3) {$\vdots$};
    \node at (-4,-7.3) {$\vdots$};
    \draw [line width=1pt, fill=black] (0,-5) circle (0.12);
    \draw [line width=1pt, fill=black] (0,-4) circle (0.12);
    \draw [line width=1pt, fill=black] (-1,-4) circle (0.12);
    \draw [line width=1pt, fill=black] (-1,-5) circle (0.12);
    \node at (0.7,-5.5) {$\textbf{0}=z_{1}$};
    \node at (-1.2,-5.5) {$-1$};
    \node at (-1.5,-3.5) {$-1+i$};
    \node at (0.3,-3.5) {$i$};
    \node at (-5,-4) {$\mathbb{C}$};
    \node at (-2,-4.5) {$P$};
    \node at (3,-4.6) {$a_{1}$};
    \node[red] at (-5.3,-5.6) {$a_{2}$};
    %\node[red] at (-4.5,-5.3) {$\gamma_{ a^{2}}$};
    \node[blue] at (-5.3,-6.6) {$a_{\ell}$};
     %\node[blue] at (-4.5,-7.4) {$\gamma_{ a^{\ell}}$};
    
    %%%%%%%%%%%%%%%%%%%%%%%%%%%map
    \node at (0,-1.5) {$\vdots$};
    \node at (-0.5,-2.4) {$\downarrow$};
    \node at (-1,-2.4) {$\pi$};
    
    %%%%%%%%%%%%%%%%%%%%%%%%%%%555
    %%%%%%%%%%%%%%%%%%%%%%%%%%%%% Sheet L_{-1}
    \draw [line width=1pt]  (-6,-1) -- (6,-1);
	\draw [line width=1pt]  (-6,4) -- (6,4);
    \draw [line width=1pt]  (-6,4) -- (-6,1.1);
    \draw [line width=1pt]  (-6,0.9) -- (-6,0.1);
    \draw [line width=1pt]  (-6,-0.1) -- (-6,-1);
	\draw [line width=1pt]  (6,4) -- (6,2.1);
    \draw [line width=1pt]  (6,1.9) -- (6,-1);
    %%%%%%%%%%Square
    \draw [fill=gray!20, line width=1.0pt] (0,2) --(0,3) -- (-1,3)-- (-1,2)-- cycle;
    \draw [blue!30, line width=1.0pt] (0,2) -- (0,3);
    \node[blue!30] at (0,2.5) {$--$};
    \draw [blue!30, line width=1.0pt] (-1,2) -- (-1,3);
    \node[blue!30] at (-1,2.5) {$--$};
    \draw [red!30, line width=1.0pt] (-1,2) -- (0,2);
    \node[red!30] at (-0.5,2) {$|$};
    \draw [red!30, line width=1.0pt] (-1,3) -- (0,3);
    \node[red!30] at (-0.5,3) {$|$};
    \draw [->, >=latex, color=black, line width=1pt] plot[smooth] coordinates {(0,2)(3,2.2)(6,2.1)};
	\draw [->, >=latex, color=black, line width=1pt] plot[smooth] coordinates {(0,2)(3,1.8)(6,1.9)};
    \node at (3,2.6) {$a_{1}^{-1}$};
    \node at (3,1.4) {$a_{1}^{0}$};
                    
    \draw [line width=1pt, fill=black] (0,2) circle (0.12);
    \draw [line width=1pt, fill=black] (0,3) circle (0.12);
    \draw [line width=1pt, fill=black] (-1,3) circle (0.12);
    \draw [line width=1pt, fill=black] (-1,2) circle (0.12);
    \draw [line width=1pt, red, fill=red] (-2,1) circle (0.12);
    \draw [line width=1pt, red, fill=red] (-2,1) circle (0.12);
    \draw [line width=1pt, blue, fill=blue] (-2,0) circle (0.12);
    \draw [->, >=latex, color=red, line width=1pt] plot[smooth] coordinates {(-2,1)(-3.5,1.2)(-6,1.1)};
    \draw [->, >=latex, color=red, line width=1pt] plot[smooth] coordinates {(-2,1)(-3.5,0.8)(-6,0.9)};
    \draw [->, >=latex, color=blue, line width=1pt] plot[smooth] coordinates {(-2,0)(-3.5,0.2)(-6,0.1)};
    \draw [->, >=latex, color=blue, line width=1pt] plot[smooth] coordinates {(-2,0)(-3.5,-0.2)(-6,-0.1)};

    \node at (-4,0.6) {$\vdots$};
    \node at (-4,-0.4) {$\vdots$};
    \node at (0,1.4) {$\pi^{-1}(\textbf{0})$};
    \node at (-5,3) {$L_{-1}$};
    \node at (-2,2.5) {$P_{-1}$};
                  
%%%%%%%%%%%%%%%%%%%%%%%%%%%%%%%%%%%%%%%
%%%%%%%%%%%%%%%Sheet L_{-1}
    \draw [line width=1pt]  (-6,5) -- (6,5);
	\draw [line width=1pt]  (-6,10) -- (6,10);
    \draw [line width=1pt]  (-6,10) -- (-6,7.1);
    \draw [line width=1pt]  (-6,6.9) -- (-6,6.1);
    \draw [line width=1pt]  (-6,5.9) -- (-6,5);
	\draw [line width=1pt]  (6,10) -- (6,8.1);
    \draw [line width=1pt]  (6,7.9) -- (6,5);

    \draw [fill=gray!20, line width=1.0pt] (0,8) --(0,9) -- (-1,9)-- (-1,8)-- cycle;
    \draw [blue!60, line width=1.0pt] (0,8) -- (0,9);
    \node[blue!60] at (0,8.5) {$--$};
    \draw [blue!60, line width=1.0pt] (-1,8) -- (-1,9);
    \node[blue!60] at (-1,8.5) {$--$};
    \draw [red!60, line width=1.0pt] (-1,8) -- (0,8);
    \node[red!60] at (-0.5,8) {$|$};
    \draw [red!60, line width=1.0pt] (-1,9) -- (0,9);
    \node[red!60] at (-0.5,9) {$|$};
    \draw [->, >=latex, color=black, line width=1pt] plot[smooth] coordinates {(0,8)(3,8.2)(6,8.1)};
	\draw [->, >=latex, color=black, line width=1pt] plot[smooth] coordinates {(0,8)(3,7.8)(6,7.9)};
    \node at (3,8.6) {$a_1^0$};
    \node at (3,7.4) {$a_{1}^1$};
                    
    \draw [line width=1pt, fill=black] (0,8) circle (0.12);
    \draw [line width=1pt, fill=black] (0,9) circle (0.12);
    \draw [line width=1pt, fill=black] (-1,9) circle (0.12);
    \draw [line width=1pt, fill=black] (-1,8) circle (0.12);
    \draw [line width=1pt, red, fill=red] (-2,7) circle (0.12);
    \draw [line width=1pt, red, fill=red] (-2,7) circle (0.12);
    \draw [line width=1pt, blue, fill=blue] (-2,6) circle (0.12);
    \draw [->, >=latex, color=red, line width=1pt] plot[smooth] coordinates {(-2,7)(-3.5,7.2)(-6,7.1)};
    \draw [->, >=latex, color=red, line width=1pt] plot[smooth] coordinates {(-2,7)(-3.5,6.8)(-6,6.9)};
    \draw [->, >=latex, color=blue, line width=1pt] plot[smooth] coordinates {(-2,6)(-3.5,6.2)(-6,6.1)};
    \draw [->, >=latex, color=blue, line width=1pt] plot[smooth] coordinates {(-2,6)(-3.5,5.8)(-6,5.9)};
    \node[red] at (-5.3,1.5) {$a_{2}^{-1}$};
    \node[blue] at (-5.3,0.5) {$a_{\ell}^{-1}$};
    \node[red] at (-3,0.5) {$a_{2}^{0}$};
    \node[blue] at (-3,-0.5) {$a_{\ell}^{0}$};

    \node at (-4,6.6) {$\vdots$};
    \node at (-4,5.6) {$\vdots$};
    \node at (0,7.4) {$\pi^{-1}(\textbf{0})$};
    \node at (-5,9) {$L_{0}$};
    \node at (-2,8.5) {$P_{0}$};

%%%%%%%%%%%%%%%%%%%%%%%%%%%%%%%%%%%%%%%%%%%%%%
%%%%%%%%%%%%%%%%%%Sheet L_{0}
    \node at (0,16.7) {$\vdots$};
    \draw [line width=1pt]  (-6,11) -- (6,11);
	\draw [line width=1pt]  (-6,16) -- (6,16);
    \draw [line width=1pt]  (-6,16) -- (-6,13.1);
    \draw [line width=1pt]  (-6,12.9) -- (-6,12.1);
    \draw [line width=1pt]  (-6,11.9) -- (-6,11);
	\draw [line width=1pt]  (6,16) -- (6,14.1);
    \draw [line width=1pt]  (6,13.9) -- (6,11);

    \draw [fill=gray!20, line width=1.0pt] (0,14) --(0,15) -- (-1,15)-- (-1,14)-- cycle;
    \draw [blue!90, line width=1.0pt] (0,14) -- (0,15);
    \node[blue!90] at (0,14.5) {$--$};
    \draw [blue!90, line width=1.0pt] (-1,14) -- (-1,15);
    \node[blue!90] at (-1,14.5) {$--$};
    \draw [red!90, line width=1.0pt] (-1,14) -- (0,14);
    \node[red!90] at (-0.5,14) {$|$};
    \draw [red!90, line width=1.0pt] (-1,15) -- (0,15);
    \node[red!90] at (-0.5,15) {$|$};
    \draw [->, >=latex, color=black, line width=1pt] plot[smooth] coordinates {(0,14)(3,14.2)(6,14.1)};
	\draw [->, >=latex, color=black, line width=1pt] plot[smooth] coordinates {(0,14)(3,13.8)(6,13.9)};
    \node at (3,14.6) {$a_{1}^1$};
    \node at (3,13.4) {$a_{1}^2$};
    \node[red] at (-5.3,7.5) {$a_{2}^{0}$};
    \node[blue] at (-5.3,6.5) {$a_{\ell}^{0}$};
    \node[red] at (-3,6.5) {$a_{2}^{1}$};
    \node[blue] at (-3,5.5) {$a_{\ell}^{1}$};
                    
    \draw [line width=1pt, fill=black] (0,14) circle (0.12);
    \draw [line width=1pt, fill=black] (0,15) circle (0.12);
    \draw [line width=1pt, fill=black] (-1,15) circle (0.12);
    \draw [line width=1pt, fill=black] (-1,14) circle (0.12);
    \draw [line width=1pt, red, fill=red] (-2,13) circle (0.12);
    \draw [line width=1pt, red, fill=red] (-2,13) circle (0.12);
    \draw [line width=1pt, blue, fill=blue] (-2,12) circle (0.12);
    \draw [->, >=latex, color=red, line width=1pt] plot[smooth] coordinates {(-2,13)(-3.5,13.2)(-6,13.1)};
    \draw [->, >=latex, color=red, line width=1pt] plot[smooth] coordinates {(-2,13)(-3.5,12.8)(-6,12.9)};
    \draw [->, >=latex, color=blue, line width=1pt] plot[smooth] coordinates {(-2,12)(-3.5,12.2)(-6,12.1)};
    \draw [->, >=latex, color=blue, line width=1pt] plot[smooth] coordinates {(-2,12)(-3.5,11.8)(-6,11.9)};
    \node[red] at (-5.3,13.5) {$a_{2}^{1}$};
    \node[blue] at (-5.3,12.5) {$a_{\ell}^{1}$};
    \node[red] at (-3,12.5) {$a_{2}^{2}$};
    \node[blue] at (-3,11.5) {$a_{\ell}^{2}$};

    \node at (-4,12.6) {$\vdots$};
    \node at (-4,11.6) {$\vdots$};
    \node at (0,13.4) {$\pi^{-1}(\textbf{0})$};
    \node at (-5,15) {$L_{1}$};
    \node at (-2,14.5) {$P_{1}$};

	\end{scope}
			\end{tikzpicture}
		\end{center}
		\caption{\emph{Constructing a Loch Ness monster with infinite-angle cone points.}}
		\label{Figure_infinite_covering}
	\end{figure}
    
We consider the following subsets of $\widetilde{S}$:

\begin{itemize}
\item[\textbf{$\bullet$}] The inverse image $\pi^{-1}(P)=\bigcup\limits_{n\in\mathbb{Z}}P_{n}\subset \widetilde{S}$ is the union of infinitely many isometric copies of $P$ with common point the vertex $\pi^{-1}(\textbf{0})$.  

\item[\textbf{$\bullet$}] For each $\ell$ with $2\leq \ell\leq N$, the inverse image $\pi^{-1}(a_{\ell})$ consists infinitely many disjoint isometric copies of $a_{\ell}$, which are parallel. See Figure \ref{Figure_infinite_covering}.
\end{itemize} 

\vspace{2mm}
 From $\widetilde{S}-\{\pi^{-1}(\textbf{0})\}$ remove the interior and the vertices of each square $P_{n}$, and then glue the opposite sides of each $P_n$ by translation. 
 Moreover, for each $\ell$, $2\leq \ell\leq N$, remove the points in $\pi\inv(z_\ell)$, cut the surface along each ray in $\pi^{-1}(a_{\ell})$, and then glue the top boundary of the ray in sheet $L_n$ with the bottom boundary of the ray in sheet $L_{n-1}$ by translation. Denote by $S$ the resulting translation surface shown in Figure \ref{Figure_infinite_covering}. The metric completion $\widehat{S}$ is the union of $S$ with the singularities of infinite cone angle:

\begin{itemize}
\item[\textbf{$\bullet$}] All points of the fibers
\[
\pi^{-1}(\textbf{0}),\; \pi^{-1}(-1),\;\pi^{-1}(-1+i), \text{ and } \pi^{-1}(i)
\]
are identified in such a way that they define a cone point of infinite cone angle, denoted by $[\pi^{-1}(\textbf{0})]$. 

\item[\textbf{$\bullet$}] For each $2\leq \ell \leq N$, the points of the fiber $\pi^{-1}(z_{\ell})$ are identified in such a way that they define a cone point of infinite cone angle, denoted by $[\pi^{-1}(z_\ell)]$.
\end{itemize}
This ensures that the tame translation surface $S$ has exactly $N$ cone points of infinite cone angle. 
To complete the proof, we show that the surface $S$ is topologically equivalent to the Loch Ness monster.

\emph{First we show the surface $S$ is one-ended}. Let $K\subset S$ be compact. We will show that there exists a compact subset $K'\subset S$ containing $K$ such that $S\setminus K'$ is connected.

Let $T$ be punctured surface with boundary obtained from $\widetilde{S}-\{\pi^{-1}(\textbf{0})\}$ by removing the vertices and the interior of each square $P_{n}$ for $n\in\mathbb{Z}$, by slitting the surface along the infinitely many rays in the inverse image $\pi^{-1}(a_\ell)$ for each $\ell \in \{2,\ldots, N\}$, and by removing the initial points $\bigcup_{\ell=2}^N \pi\inv(z_\ell)$ of the rays.

Note that there is a natural ``slit-regluing'' map $p:T\to \widetilde S-\{\pi^{-1}(\textbf{0})\}$, which is the identity at points that are not on the boundaries corresponding to the rays $\pi^{-1}(a_\ell)$ and which identifies by translation the top and bottom boundaries of the lift of $a_\ell$ in sheet $L_n$, for each $2\leq \ell\leq N$ and $n\in \mathbb{Z}$. While the map $p$ is not a surjective projection onto $\widetilde S-\{\pi^{-1}(\textbf{0})\}$, the image $p(T)$ is isometric to the subsurface of $\widetilde{S}-\{\pi^{-1}(\textbf{0})\}$ obtained by removing the points $\bigcup\limits_{\ell=2}^{N}\{\pi^{-1}(z_\ell)\}$ and the interior and vertices of each square $P_{n}$ for $n\in\mathbb{Z}$.

Note that the surface $S$ as described above is a quotient of $T$. Denote by $q:T\to S$ this projection. We note that $q$ is a proper map, and so the inverse image $q^{-1}(K)$ is a compact subset of $T$. 

Thus we have that $p\left(q^{-1}(K)\right)$ is a compact subset of $\widetilde S-\{\pi^{-1}(\textbf{0})\}$. Since $\widetilde S-\{\pi^{-1}(\textbf{0})\}$ has only one end, there exists a compact subset $K_{0}\subset \widetilde S-\{\pi^{-1}(\textbf{0})\}$ such that $p\left(q^{-1}(K)\right)\subset K_{0}$ and $C_{0}:=\left(\widetilde S-\{\pi^{-1}(\textbf{0})\}\right)-K_{0}$ is connected. 

We observe that the compact set $K_0$ may contain points in the lifts $\pi\inv(z_\ell)$ as well as vertices or points in the interiors of the squares $P_n$. These points are not in the image $p(T) \subset \widehat S$. As we would like to pass $K_0$ through $q \circ p\inv$, we will take a closed subset of $K_0$ containing $p\left(q^{-1}(K)\right)$ and avoiding these points.

Given that the set of inverse images $\{\pi^{-1}(z_\ell): 2\leq \ell\leq N\}$ is discrete and $K_{0}$ is compact, only finitely many of these points belong to $K_{0}$. By taking a closed subset of $K_0$ if necessary, we can assume that none of these points are in the boundary $\partial K_{0}$. Consequently, only finitely many rays in $\pi^{-1}(a_\ell)$ intersect $K_{0}$ for each $2\leq \ell\leq N$. Denote all these rays by $a'_{1},\dots, a'_{m}$, and their respective initial points by $z'_{1},\ldots,z'_{m}$. 
A similar argument tells us that $K_{0}$ intersects the vertices or interior of only finitely many of the squares $P_n$. Denote these squares by $P'_{1},\ldots, P'_{t}$.
Let $\varepsilon > 0$ be such that
\begin{itemize}
    \item the open balls $B_{\varepsilon}(z'_{1}),\ldots, B_{\varepsilon}(z'_{m})$ are pairwise disjoint and contained in $K_{0}-p\left(q^{-1}(K)\right)$,
    
    \item for each vertex $v$ of $P'_{i}$ except $\pi^{-1}(\textbf{0})$ for $1\leq i\leq t$, the open ball $B_{\varepsilon}(v)$ is contained in $K_{0}-p\left(q^{-1}(K)\right)$, and 

    \item the open ball $B_{\varepsilon}(\pi^{-1}(\textbf{0}))$ does not intersect the compact $p\left(q^{-1}(K)\right)$ and is contained in $\left(\widetilde S- \{\pi^{-1}(\textbf{0})\}\right)-K_{0}$.
\end{itemize}

We now consider the perforated compact set
\[
K_{0}':=K_{0}-\left( \bigcup\limits_{i=1}^{m}B_{\varepsilon}(z'_{i})\cup \bigcup\limits_{1 \leq j \leq t} \text{int}(P_j') \cup \bigcup_{\substack{v \text{ vertex of } P'_{j}\\ 1\leq j \leq t}}B_{\varepsilon}(v)\right)\subset K_{0}.
\]
We have the following facts:
\begin{itemize}
\item[\textbf{$\bullet$}] The complement $(\widetilde S-\{\pi^{-1}(\textbf{0})\})-K_{0}'$ is composed of $m+1$ connected components, 

\[C_{0}=\left(\widetilde S-\{\pi^{-1}(\textbf{0})\}\right)-K_{0}, \text{ and } C_{1}=B_{\varepsilon}(z'_{1}), \ldots, C_{m}=B_{\varepsilon}(z'_{m}).\] 

\item[$\bullet$] Given that $K_{0}'$ intersects a finite numbers of rays, $a_{1}',\ldots,a_{m}'$, then the inverse image $p^{-1}(K_{0}')$ is a compact subset of $T$. This compact $p^{-1}(K_{0})$ is the set one obtains from $K_{0}'$ by cutting along the curve $K_{0}'\cap a_{i}'$, for $1\leq i\leq m$.

\item[$\bullet$] By construction, the image $q(p^{-1}(K_{0}))=K'$ is a compact subset of $S$ such that $K\subset K'$. 

\end{itemize}

We claim that $S - K'$ is connected and note that the complement $S-K'$ is equal to the quotient space obtained by restricting the slit-and-glue operation to the union of the sets $C_0, \ldots, C_{m}$, after removing the initial points $\pi^{-1}(z_\ell)$ and the vertices of the squares $P_{n}$, $n\in \Z$. We denote by $\tilde{C}_{i}$ the connected quotient space obtained from $C_{i}$ by applying the slit-and-glue operation, where $0\leq i\leq m$.

We note that the punctured open ball $U_{0}=B_{\varepsilon}([\pi^{-1}(\textbf{0})])-\{[\pi^{-1}(\textbf{0})]\}$ of $S$ is connected and belongs to some connected component of $S-K'$. By construction, this punctured open ball intersects the connected set $\tilde{C}_{0}$, and so the union $U_{0}\cup \tilde{C}_{0}$ is connected. 
%Moreover, the quotient spaces  $\tilde{C}_{m+1},\ldots,\tilde{C}_{m+3t}$ are contained in $U_{0}$. This implies that the union $U_{0}\cup\tilde{C}_{0}\cup\bigcup\limits_{i=m+a}^{m+3t}\tilde{C}_{i}$  is connected.

For each $1\leq j\leq m$, the connected punctured ball $U_{j}=B_{\varepsilon}([\pi^{-1}(z'_{j})])-\{[\pi^{-1}(z'_{j})]\}$ is contained in some connected component of $S-K'$. By construction, $U_{j}$ also intersects $\tilde{C}_{0}$ and contains the set $\tilde{C}_{j}$. Thus, the union $\tilde{C}_{0}\cup \bigcup\limits_{j=1}^{m}U_{j}\cup\bigcup\limits_{j=1}^{m}\tilde{C}_{i} = S-K'$ is connected. 

%Given that the connected $U_{0}\cup\tilde{C}_{0}\cup\bigcup\limits_{i=m+a}^ {m+3t}\tilde{C}_{i}$ and $\tilde{C}_{0}\cup \bigcup\limits_{j=1}^{m}U_{j}\cup\bigcup\limits_{j=1}^{m}\tilde{C}_{i}$ have in common the set $\tilde{C}_{0}$, then the union $U_{0}\cup\tilde{C}_{0}\cup\bigcup\limits_{i=m+a}^ {m+3t}\tilde{C}_{i}\cup \tilde{C}_{0}\cup \bigcup\limits_{j=1}^{m}U_{j}\cup\bigcup\limits_{j=1}^{m}\tilde{C}_{i}=S-K'$ is connected.

\emph{Finally, we prove that $S$ has infinite genus.}  The  tame translation surface $S_{1}$ described below is necessary to prove that $S$ has infinite genus.

Let $S_{1}$ be the tame translation surface obtained from the complex plane $\mathbb{C}$ by removing the interior of the unit square $P$, and gluing its parallel sides by translations. We then have the following fact

\begin{itemize}
\item[\textbf{(1)}]The surface $S_{1}$ has genus one and only one end. 

\item[\textbf{(2)}] By construction, the map $\pi:S\to \mathbb{C}$ induces an infinite cyclic covering $\Pi:\widehat{S}\to S_{1}$ branched over the points $z_{1},z_{2},\ldots,$ and $z_{N}$.
That is, for each $\ell$, with $1\leq \ell\leq N$ the inverse image $\Pi^{-1}(z_{\ell})$ is a singularity of infinite cone angle. 
\end{itemize}

Let $U$ be a connected open subset of the complex plane $\mathbb{C}$ containing the unit square $P$ such that just its corner $z_{1}=\textbf{0}$ is on the boundary $\partial U$ of $U$. See Figure \ref{fig:open-minus-square}. We can assume that the closure $\overline{U}$ of $U$ does not contain points in $\{z_{2},z_{3},\ldots, z_{N}\}$. Let $\widetilde{U}$ be the connected open subset of $S_{1}$, which comes from $U$ by removing the interior of the unit square $P$ (including its vertices,) and gluing pairs of parallel sides by a translation as shown Figure \ref{fig:open-minus-square}. The following holds:
\begin{itemize}
    \item[\textbf{$\bullet$}] The connected open subset $\widetilde{U}\subset S_{1}$ is homeomorphic to a torus with a hole. See Construction \ref{cons:torus_with_boundary}

    \item[\textbf{$\bullet$}] On the boundary $\partial \widetilde{U}$ of $\widetilde{U}$ contains a unique infinite-angle singular point. Moreover, $\widetilde{U}$ does not contain singular points in its interior.
\end{itemize}

We remark that the inverse image $\Pi^{-1}(\widetilde{U})$ is the union of infinitely many copies of $\widetilde{U}$, which are pairwise disjoint. This ensures that $S$ has infinite genus.

The preceding construction establishes the existence of a tame translation structure on the Loch Ness monster with singularity data $\mathbf{k}_{\infty}=(k_{n})_{n\in\mathbb{N}_{0}}$ satisfying $k_{0}\neq 0$ and $k_{n}=0$ for every $n\geq 1$. 
\qed

\begin{remark}\label{remark:LNM_strata_1_0_0_0}
The surface $S$ described in the proof above can be modified slightly to produce a large family of tame translation structures on the Loch Ness monster surface. Each of these structures has a unique singularity of infinite angle, and no two of them are isomorphic. 

The construction described above admits a natural generalization. Instead of removing copies of the unit square from the infinite cyclic branched covering $\widetilde{S}$ of $\mathbb{C}$ branched over $\mathbf{0}$, one may remove from each sheet $L_{n}$ of $\widetilde{S}$ the interior of a suitable polygon $P_{n}$ (not necessarily the same for every $n\in\mathbb{Z}$), chosen so that one of its vertices coincides with the unique singular point of $L_n$. The parallel sides of $P_{n}$ are then identified in pairs by translations.

By suitable we mean that after performing these side identifications, the polygon $P_n$ yields either a flat torus $\mathbb{T}_{\tau}$ or a compact translation surface $S_{g}$ of genus $g\geq 2$ with a single conical singularity. The flat torus $\mathbb{T}_{\tau}$ is in the modular space of flat structures $\mathcal{H}_{1}(1):=\mathbb{H}/{\rm SL}_{2}(\mathbb{Z})$, where $\mathbb{H}$ denotes the upper half-plane, and the compact translation surface $S_{g}$ with genus $g\geq 2$ is in the stratum $\mathcal{H}_{g}(2g-2)$. 

By carrying out this procedure on every sheet $L_{n}$, one obtains a large family of tame translation structures on the Loch Ness monster surface, each having a unique singularity of infinite angle. Moreover, by varying the collection of polygons $P_{n}$ with ${n\in\mathbb{Z}}$, one can construct infinitely many pairwise non-isomorphic such structures.

This construction shows that the {non-compact stratum} $\mathcal{H}(S,\mathbf{k}_{\infty})$, where $S$ denotes the Loch Ness monster and $\mathbf{k}_{\infty}=(1,0,0,\ldots)$, contains a subspace naturally identified with the infinite product
$\prod\limits_{n\in\mathbb{Z}} X_{n}$,
where each $X_{n}$ is a copy of the disjoint union of strata
$\bigsqcup_{g\geq 1}\mathcal{H}_{g}(2g-2)$.
Indeed, for every sheet $L_{n}$ one may independently choose a compact translation surface represented by a point of $\mathcal{H}_{g}(2g-2)$, and the resulting choices determine a tame translation structure on $S$ with a unique singularity of infinite angle. That is,
\[
\prod\limits_{n\in\mathbb{Z}} X_{n}\subset \mathcal{H}(S,\mathbf{k}_{\infty}).
\]
\end{remark}

\subsection{Proof Theorem \ref{theorem_case_2_MTB}}\label{proof_theorem_Case_2_MTB} ($k_0=0$ and $\sum\limits_{n=1}^\infty k_n = \infty$)

Let $\kinf$ be such that $k_0=0$ and $\sum\limits_{n=1}^\infty k_n = \infty$. We want to construct a translation structure on the Loch Ness monster with no cone points of infinite cone angle and with an infinite number of cone points of finite cone angle, as prescribed by the counts $k_1, k_2, k_3,\ldots$. We remind the reader that for $n\geq 1$, the entry $k_n$ in $\kinf$ gives the number of cone points of angle $2\pi(n+1)$. Let $L_0 = \C$, a copy of the complex plane. For each $n\geq 1$, we will perform Operation 1, 2, or 3 to get a discrete set of exactly $k_n$ cone points of angle $2\pi(n+1).$ See Figure~\ref{fig:finite_singularities} for an example.

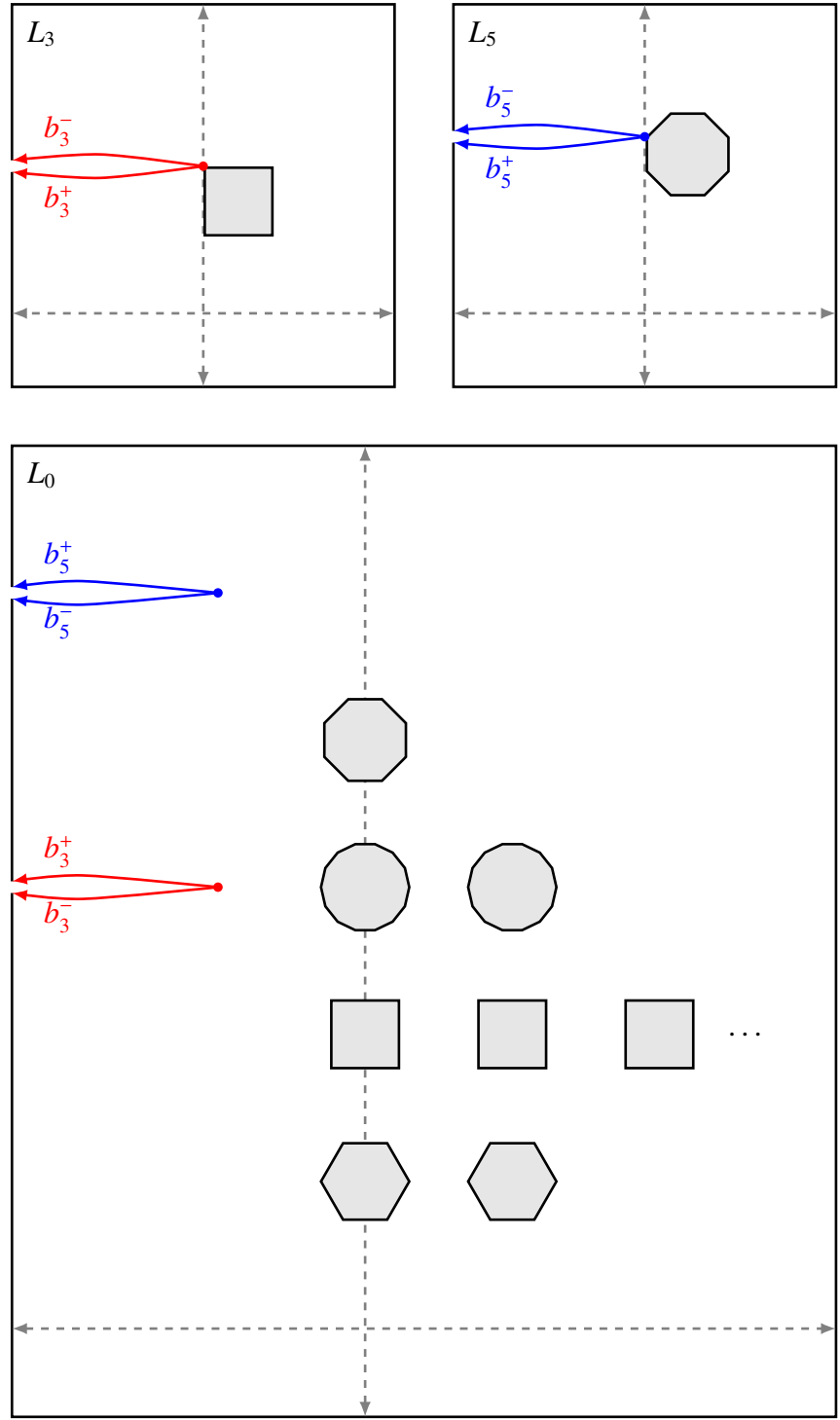
\begin{figure}[!ht]
		\begin{center}
			\begin{tikzpicture}[scale=.8]
    %%%%%%%%%%%%%%%%%%%%%%%%%%%%%%%%%%%%%%%%%%%%%%%
    %%%%%%%%% SHEET L_0
	\draw [line width = 1pt] (-8, 2.9) -- (-8, -6) -- (6, -6) -- (6, 10.5) -- (-8, 10.5) -- (-8, 8.1);

    \draw [line width = 1pt] (-8, 3.1) -- (-8, 7.9);

    \draw [black!50, line width=1pt,  <->, >=latex, dashed] (-8,-4.5) -- (6, -4.5);
    \draw [black!50, line width=1pt,  <->, >=latex, dashed] (-2,-6) -- (-2,10.5);  
    \node at (-7.5,10) {$L_{0}$};

    %%%%% HEXAGONS
    \node[draw, line width=1pt, regular polygon, regular polygon sides=6, minimum size=1.2cm, fill=gray!20] at (-2,-2) {};

    \node[draw, line width=1pt, regular polygon, regular polygon sides=6, minimum size=1.2cm, fill=gray!20] at (.5,-2) {};

    %%%%% SQUARES
    \node[draw, line width=1pt, regular polygon, regular polygon sides=4, minimum size=1.3cm, fill=gray!20] at (-2,.5) {};

    \node[draw, line width=1pt, regular polygon, regular polygon sides=4, minimum size=1.3cm, fill=gray!20] at (.5,.5) {};

    \node[draw, line width=1pt, regular polygon, regular polygon sides=4, minimum size=1.3cm, fill=gray!20] at (3,.5) {};

    \node at (4.5, .5) {$\dots$};

    %%%%% 14-gons
    \node[draw, line width=1pt, regular polygon, regular polygon sides=14, minimum size=1.2cm, fill=gray!20] at (-2,3) {};

    \node[draw, line width=1pt, regular polygon, regular polygon sides=14, minimum size=1.2cm, fill=gray!20] at (.5,3) {};

    \draw [->, >=latex, line width=1pt, red] plot[smooth] coordinates {(-4.5,3)(-6.8,3.2)(-8,3.1)};
    \node[red] at (-7.2,3.6) {$b_{3}^{+}$};

    \draw [->, >=latex, line width=1pt, red] plot[smooth] coordinates {(-4.5,3)(-6.8,2.8)(-8,2.9)};
    \node[red] at (-7.2,2.5) {$b_{3}^{-}$};
    %\node[draw, line width=1pt, regular polygon, regular polygon sides=4, minimum size=1.3cm, fill=gray!20] at (-4.9,2.4) {};
    \fill[red] (-4.5,3) circle (0.08);

    %%%%% OCTAGONS
    \node[draw, line width=1pt, regular polygon, regular polygon sides=8, minimum size=1.2cm, fill=gray!20] at (-2,5.5) {};

    %%%%% 22-gons

    \draw [->, >=latex, line width=1pt, blue] plot[smooth] coordinates {(-4.5,8)(-6.8,8.2)(-8,8.1)};
    \node[blue] at (-7.2,8.6) {$b_5^+$};
    
    \draw [->, >=latex, line width=1pt, blue] plot[smooth] coordinates {(-4.5,8)(-6.8,7.8)(-8,7.9)};
    \node[blue] at (-7.2,7.5) {$b_5^-$};
    
    \fill[blue] (-4.5,8) circle (0.08);

%%%%%%%%%%%%%%%%%%%%%%%%%%%%%%%%%%%%%%%%%%%%%%
%%%%%%%%%%%%%%%%%%Sheet L_{3}
    \draw[line width = 1] (-8,15.15) -- (-8, 11.5) -- (-1.5, 11.5) -- (-1.5, 18) -- (-8, 18) -- (-8, 15.35);

    \draw [black!50, line width=1pt,  <->, >=latex, dashed] (-1.5,12.75) -- (-8, 12.75);
    \draw [black!50, line width=1pt,  <->, >=latex, dashed] (-4.75,11.5) -- (-4.75,18);  
    \node at (-7.5,17.5) {$L_{3}$};

    \draw [->, >=latex, line width=1pt, red] plot[smooth] coordinates {(-4.75,15.25)(-6.5,15.45)(-8,15.35)};
    \node[red] at (-7.2,15.85) {$b_3^-$};
    
    \draw [->, >=latex, line width=1pt, red] plot[smooth] coordinates {(-4.75,15.25)(-6.5,15.05)(-8,15.15)};
    \node[red] at (-7.2,14.65) {$b_3^+$};
    
    \node[draw, line width=1pt, regular polygon, regular polygon sides=4, minimum size=1.3cm, fill=gray!20] at (-4.15,14.65) {};
    \fill[red] (-4.75,15.25) circle (0.08);

%%%%%%%%%%%%%%%%%%%%%%%%%%%%%%%%%%%%%%%%%%%%%%
%%%%%%%%%%%%%%%%%%Sheet L_{5}
   \begin{scope}[xshift = 7.5cm]
    
    \draw[line width = 1] (-8,15.65) -- (-8, 11.5) -- (-1.5, 11.5) -- (-1.5, 18) -- (-8, 18) -- (-8, 15.85);

    \draw [black!50, line width=1pt,  <->, >=latex, dashed] (-8,12.75) -- (-1.5, 12.75);
    \draw [black!50, line width=1pt,  <->, >=latex, dashed] (-4.75,11.5) -- (-4.75,18);  
    \node at (-7.5,17.5) {$L_{5}$};

    \draw [->, >=latex, line width=1pt, blue] plot[smooth] coordinates {(-4.75,15.75)(-6.5,15.95)(-8,15.85)};
    \node[blue] at (-7.2,16.35) {$b_5^-$};
    
    \draw [->, >=latex, line width=1pt, blue] plot[smooth] coordinates {(-4.75,15.75)(-6.5,15.55)(-8,15.65)};
    \node[blue] at (-7.2,15.15) {$b_5^+$};
    
    \node[draw, line width=1pt, regular polygon, regular polygon sides=8, minimum size=1.2cm, fill=gray!20] at (-4.02,15.45) {};
    \fill[blue] (-4.75,15.75) circle (0.08);

    \end{scope}

			\end{tikzpicture}
		\end{center}
		\caption{\em A Loch Ness monster corresponding to the sequence $(0;4, \infty, 5, 1, 1, 0, 0, \ldots)$. We observe that the polygons removed to the right of $b_3^\pm$ are 14-gons.}
		\label{fig:finite_singularities}
	\end{figure}

\begin{itemize}
    
    \item Operation 1: $n$ even and $k_n \in \N \cup \{\infty\}$. For each $j \in \{0, \ldots, k_n-1\}$ (or $j \in \N_0$ if $k_n = \infty$), let $P_j$ be a unit-area regular $2n$-gon centered at the point $2j + i(2n+1)$ in $L_0$. Remove the interior of each $P_j$ and identify opposite edges via translation. Since $n$ is even, by part (a) of Lemma~\ref{lem:2n-gon}, this process introduces $k_n$ singularities of cone angle $(2n+2)\pi = 2\pi(n+1)$, as desired.

    \item Operation 2: $n$ odd and $k_n = 2m$ finite and even or $k_n = \infty$. For each $j \in \{0, \ldots, m-1\}$ (or $j \in \N_0$ if $k_n = \infty$), let $P_j$ be a unit-area regular $2(2n+1)$-gon centered at the point $2j + i(2n+1)$ in $L_0$. Remove the interior of each $P_j$ and identify opposite edges via translation. Since $2n+1$ is odd, by part (b) of Lemma~\ref{lem:2n-gon}, this process introduces $k_n$ singularities of cone angle $(2n+2)\pi = 2\pi(n+1)$, as desired.

    \item Operation 3: $n$ odd and $k_n = 2m+1$ finite and odd. As in Case 2, for each $j \in \{0, \ldots, m-1\}$, let $P_j$ be a unit-area regular $2(2n+1)$-gon centered at the point $2j + i(2n+1)$ in $L_0$. Remove the interior of each $P_j$ and identify opposite edges via translation to produce $2m$ singularities of cone angle $(2n+2)\pi = 2\pi(n+1)$. Additionally, introduce an infinite slit parallel to the negative real axis with initial point $p = -2 + i(2n+1)$ with top edge $b_n^+$ and bottom edge $b_n^-$ as in Figure~\ref{fig:finite_singularities}. We also introduce a new sheet $L_n = \C$ with a corresponding slit parallel to the negative real axis and initial point $q = i(2n+1)$. After gluing the boundaries of the slits, the point $[p] = [q]$ has cone angle $4\pi$. In $L_n$, let $Q$ be a regular unit-area $2(n-1)$-gon which intersects the slit $b_n^\pm$ only at the initial point $q$. We remove the interior of $Q$ and identify opposite edges. By part (a) of Lemma~\ref{lem:2n-gon-on-sing}, after this operation the singularity $[p]=[q]$ has cone angle $4\pi + 2(n-1)\pi = 2\pi(n+1)$. Thus we have $2m+1$ singularities of cone angle $2\pi(n+1)$, as desired.
\end{itemize}

It remains to be shown that the resulting surface is homeomorphic to the Loch Ness monster. First we check that the resulting surface has infinite genus. Given that $\sum\limits_{n=1}^\infty k_n = \infty$, we know that $k_n = \infty$ for some $n \geq 1$ or $k_n > 0$ for infinitely many $n$. Suppose $k_n = \infty$ for some $n$. Then in performing Operation 1 or 2, depending on the parity of $n$, we remove the interiors of infinitely many polygons and identify opposite edges. By Lemma~\ref{lem:2n-gon}, this gives us a surface of infinite genus.

On the other hand, suppose that $k_n > 0$ for infinitely many $n \in \N$. We note that by performing Operation 1, 2, or 3 for a finite $k_n$, we increase the genus of the surface by at least one. Since we perform infinitely many of these operations, the genus of the resulting surface is infinite.

Finally, we check that the surface is one-ended. We start with $L_0 = \C$, which is one-ended, and Operations 1, 2, and 3 involve two types of surgeries. First, we remove the interiors of polygons and identify opposite edges. In $\C$, removing a bounded set does not affect the end space, and the identification of opposite edges is a proper quotient map. Therefore, this type of procedure does not affect the number of ends.
The second type of surgery happens in Operation 3, introducing a new sheet $L_n$ and gluing it to $L_0$ along an infinite slit. Here we apply Lemma~\ref{lemma:one_ends_glue_case_I} inductively to conclude that the resulting surface has one end.
\qed

\subsection{Proof Theorem \ref{theorem_case_3_MTB}} ($k_0 \neq 0$ and $\sum\limits_{n=1}^\infty k_n > 0$)
 
A tame translation structure on the Loch Ness monster with singularity data $\textbf{k}_{\infty}$ can be obtained as follows. Consider the surface S defined in the proof of Theorem~\ref{theorem_case_1_MTB} (see \S \ref{proof_theorem_Case_1_MTB}), and then apply the construction described in the proof of Theorem \ref{theorem_case_2_MTB} (see \S \ref{proof_theorem_Case_2_MTB}), starting with the sheet $L_0$.
\qed

%%%%%%%%%%%%%%%%%%%%%%%%%%%%%%%
%%%%%%%%%%%%%%%%%
\section{Future / open questions}\label{sec:future}
The authors suspect that Theorem \textbf{A} should generalize to other infinite-type surfaces. In particular, the methods used to prove Theorem \textbf{A} should work for surfaces where ${\rm End}(S)={\rm End}_{\infty}(S)$ since these surfaces have pants decompositions without the need for cylinders.
\begin{conjecture}
 Let $S$ be an orientable non-compact surface with end space  ${\rm End}(S)={\rm End}_{\infty}(S)$. If $S$ admits a tame translation structure without infinite cone angle singularities, then $S$ has infinitely many finite cone angle singularities.
\end{conjecture}

We can also ask how much of the moduli space of translation structures on the Loch Ness monster we see with our constructions. Motivated by Remark \ref{remark:LNM_strata_1_0_0_0}, we have the following question.

\begin{question}
Let $\mathcal{H}(S,\mathbf{k}_{\infty})$ denote the stratum of the Loch Ness monster $S$ with singularity data $\mathbf{k}_{\infty}$ as in Theorems~\ref{theorem_case_1_MTB}, \ref{theorem_case_2_MTB}, and \ref{theorem_case_3_MTB}. Do generalizations of our constructions give the whole moduli space? That is, for example, is it true that for $\kinf = (1,0,0,\ldots)$ we have
\[
\mathcal{H}(S,\mathbf{k}_{\infty})=\prod\limits_{n\in\mathbb{Z}}X_{n},
\quad\text{ where }\quad
X_n = \bigsqcup_{g\geq 1}\mathcal{H}_g(2g-2)\text{?}
\]
In other words, can every tame translation structure on the Loch Ness monster surface with a unique singularity of infinite cone angle be obtained from the construction described in Remark \ref{remark:LNM_strata_1_0_0_0}?
\end{question}

\end{document}